\documentclass{article}

\newcommand \provXc[3]{[#1:#2]_{#3}} 
\def\leqLam{\leq_\Lambda}
\newcommand \consXc[3]{\langle #1:#2\rangle_{#3}} 
\newcommand\leLam{<_\Lambda}
\newcommand{\rca}{\ensuremath{{\mathrm{RCA}}_0}\xspace}
\def\sentences{{\mathcal S}^1_\omega}

\usepackage{amsmath,amssymb,amsthm,units,stmaryrd}
\usepackage{qtree,bussproofs}

\usepackage{yfonts}
\usepackage{units,stmaryrd}
\usepackage[colorlinks]{hyperref}
\usepackage[usenames,dvipsnames]{color}
\hypersetup{
  colorlinks,
  citecolor=Violet,
  linkcolor=Red,
  urlcolor=Blue}
\usepackage{graphicx}
\usepackage[all]{xy}
\newtheorem{theorem}{Theorem}[section]

\newtheorem{definition}[theorem]{Definition}
\newtheorem{lemma}[theorem]{Lemma}

\newtheorem{corollary}[theorem]{Corollary}
\newtheorem{proposition}[theorem]{Proposition}

\newtheorem{observation}[theorem]{Observation}

\newcommand{\logic}[1]{{\ensuremath{\mathbf{#1}}}\xspace}

\newcommand{\provx}[2]{[{#1}]_{#2}}
\newcommand{\consx}[2]{\langle{#1}\rangle_{#2}}

\newcommand{\glp}{{\ensuremath{\mathsf{GLP}}}\xspace}

\usepackage{xspace}

\newcommand{\pa}{\ensuremath{{\mathrm{PA}}}\xspace}

\newcommand{\zfc}{\ensuremath{{\mathrm{ZFC}}}\xspace}

\newcommand{\eca}{\ensuremath{{{\rm ECA}_0}}\xspace}
\newcommand{\AcaNaught}{\ensuremath{{{\rm ACA}_0}}\xspace}

\newcommand{\gl}{\logic{GL}}

\newcommand{\ea}{\ensuremath{{\mathrm{EA}}}\xspace}

\newcommand{\base}{\ensuremath{{\mathrm{ACA}}_0}\xspace}
\newcommand{\la}{\langle}
\newcommand{\ra}{\rangle}

\def\fmodels{\xymatrix{
\ar@{|=}[r]^{<\omega}&
}
}
\def\nmodels{\xymatrix{
\ar@{|=}[r]^{N}&
}
}
\def\<{\left <}

\def\nc{{\Box}}

\def\mlang{{\mathcal L}_\nc}
\def\alang{{\sf L}_{2}}

\def\>{\right >}

\DeclareSymbolFont{AMSb}{U}{msb}{m}{n}
\DeclareMathSymbol{\N}{\mathbin}{AMSb}{"4E}
\DeclareMathSymbol{\Z}{\mathbin}{AMSb}{"5A}
\DeclareMathSymbol{\R}{\mathbin}{AMSb}{"52}
\DeclareMathSymbol{\Q}{\mathbin}{AMSb}{"51}
\DeclareMathSymbol{\I}{\mathbin}{AMSb}{"49}
\DeclareMathSymbol{\C}{\mathbin}{AMSb}{"43}

\newcommand{\bt}{\begin{theorem}}
\newcommand{\et}{\end{theorem}}
\newcommand{\bl}{\begin{lemma}}
\newcommand{\el}{\end{lemma}}

\newcommand\remove[1]{}

\usepackage{breqn,bm, amssymb}
\usepackage{lipsum} 
\usepackage{authblk}
\usepackage{hyperref}

\newcommand*{\email}[1]{%
    \normalsize\href{mailto:#1}{#1}\par
   }

\def\alang{{{\bm \Pi}^1_\omega}}
\newcommand\GuardLang[1]{\alang\upharpoonright #1}

\newcommand{\boxBox}[1]{[#1]^{\Box}}
\newcommand{\boxDiamond}[1]{\la #1 \ra^{\Box}}
\newcommand{\nboxBox}[1]{[#1]}
\newcommand{\nboxDiamond}[1]{\la#1\ra}

\newcommand{\vboxBox}[1]{[#1]^{\boxtimes}}
\newcommand{\vboxDiamond}[1]{\la #1 \ra^{\boxtimes}}

\newcommand{\bBox}[1]{[#1]}
\newcommand{\bDiamond}[1]{\la #1 \ra}

\begin{document}

\title{M\"unchhausen provability}


\author{Joost J. Joosten}
\affil{University of Barcelona\\
\email{jjoosten@ub.edu}}

\maketitle

\begin{abstract}
By Solovay's celebrated completeness result \cite{Solovay:1976} on formal provability we know that the provability logic \gl describes exactly all provable structural properties for any sound and strong enough arithmetical theory with a decidable axiomatisation. Japaridze generalised this result in \cite{Japaridze:1988} by considering a polymodal version \glp of \gl with modalities $[n]$ for each natural number $n$ referring to ever increasing notions of provability. 

Modern treatments of \glp tend to interpret the $[n]$ provability notion as ``provable in a base theory $T$ together with all true $\Pi^0_n$ formulas as oracles''. In this paper we generalise this interpretation into the transfinite. In order to do so, a main difficulty to overcome is to generalise the syntactical characterisations of the oracle formulas of complexity $\Pi^0_n$ to the hyper-arithmetical hierarchy. The paper exploits the fact that provability is $\Sigma^0_1$ complete and that similar results hold for stronger provability notions. As such, the oracle sentences to define provability at level $\alpha$ will recursively be taken to be consistency statements at lower levels: provability through provability whence the name of the paper. 

The paper proves soundness and completeness for the proposed interpretation for a wide class of theories; namely for any theory that can formalise the recursion described above and that has some further very natural properties. Some remarks are provided on how the recursion can be formalised into second order arithmetic and on lowering the proof-theoretical strength of these systems of second order arithmetic.
\end{abstract}

\section{Introduction}
As mentioned in the abstract, by Solovay's celebrated completeness result \cite{Solovay:1976} on provability we know that the provability logic \gl describes exactly all provable structural properties for any sound and strong enough arithmetical theory with a decidable axiomatisation. Japaridze generalised this result in \cite{Japaridze:1988} by considering a polymodal version \glp of \gl with modalities $[n]$ for each natural number $n$ referring to ever increasing notions of provability. 

Japaridze considered an arithmetical interpretation of the logic \glp where the $[n]$ referred to a natural formalisation of ``provable over the base theory $T$ using at most $n$ nested applications of the $\omega$-rule''. Beklemishev introduced in \cite{Beklemishev:2005:VeblenInGLP} the logics $\glp_\Lambda$ that are like \glp only that they now include a sequence of provability predicates $[\alpha]$ of ever increasing strength for each ordinal $\alpha$ below some fixed ordinal $\Lambda$. In \cite{FernandezJoosten:2018:OmegaRuleInterpretationGLP} the authors generalised Japaridze's result into the transfinite by providing an interpretation of $\glp_\Lambda$ for recursive $\Lambda$ into second order arithmetic by allowing for $[\alpha]$ at most $\alpha$ nestings of the omega rule, thereby providing a first arithmetical interpretation of $\glp_\Lambda$ for $\Lambda >\omega$. In a recent paper (\cite{BeklemishevPakhomov:2019:GLPforTheoriesOfTruth}) Beklemishev and Pakhomov provide an alternative interpretation in first order arithmetic enriched with a collection of ever more expressive truth predicates indexed by the ordinals.  

Modern treatments of $\glp_\omega$ tend to interpret the $[n]$ provability notion as ``provable in a base theory $T$ together with all true $\Pi^0_n$ formulas''. Let us call this the \emph{truth-interpretation} here. The main reason for the popularity of the truth-interpretation is that the resulting provability hierarchies run in phase with the arithmetical hierarchy and they imply good preservation properties between different consistency statements giving rise to the so-called \emph{reduction property}. In particular, due to these good properties Beklemishev was able to set $\glp_\omega$ to work to perform proof-theoretical analyses of Peano Arithmetic and its kin (\cite{Beklemishev:2003:ProofTheoreticAnalysisByIteratedReflection, Beklemishev:2004:ProvabilityAlgebrasAndOrdinals, Beklemishev:2005:Survey}). Below we shall give more circumstantial evidence to why the truth interpretation is optimal.

As mentioned, the first arithmetical interpretation of transfinite polymodal provability logic (\cite{FernandezJoosten:2018:OmegaRuleInterpretationGLP}) was, like Japaridze's original approach, based on iterating applications of the omega rule. Although it was observed in \cite{Joosten:2013:AnalysisBeyondFO} that soundness of the interpretation is sufficient for the purpose of an ordinal analysis, the paper also contained a completeness proof in such general lines that it can be applied to a wide range of interpretations. 

It seemed however, that the omega-rule interpretation does not have all the desirable properties to make it directly a useful tool for ordinal analyses. Even though various known fragments of second order arithmetic like $\mathrm{ATR}_0$, $\Pi^1_1-\mathrm{CA}_0$ and $\Pi^1_1-\mathrm{CA}_0 + \mbox{Bar Induction}$ can be characterised (\cite{CordonFernandezJoostenLara:2017:PredicativityThroughTransfiniteReflection, Fernandez:2015:ImpredicativeReflection}) in terms of reflection principles using versions of the omega rule interpretation of $GLP_\Lambda$, the fine-structure between various consistency statements could not be proven. 

One possible reason may be that the omega provability predicates do not tie up with the arithmetical hierarchy and Turing jumps as observed in \cite[Lemma 9]{Joosten:2015:TuringJumpsThroughProvability}. A more concrete and serious objection is given in an unpublished simple observation from Fern\'andez Duque: using only one application of the omega-rule one can prove any induction axiom so that the one-consistency of primitive recursive arithmetic in the omega-rule sense suffices to prove the consistency of Peano arithmetic.

In short, the truth interpretation of $\glp_\omega$ has better properties than the omega-rule interpretation. However, one advantage of the omega-rule interpretation is its amenability to transfinite generalisations. The formalisation of the truth interpretation relies on a syntactical characterisation of the arithmetical hierarchy in terms of the $\Sigma^0_n$ formulas. It remained unclear how to generalise this in a canonical way to the hyperarithmetical setting or beyond without extending the language in a way that often seems rather ad-hoc. 

The idea of this paper to overcome this is very simple yet turns out to be rather powerful. The \emph{Friedman-Goldfarb-Harrington} theorem (FGH) tells us that for a wide range of theories, in a sense, the canonical consistency predicate is $\Pi^0_1$ complete. Thus, instead of using a true $\Pi^0_1$ sentence as oracle for the $[1]_T$ provability predicate in the truth interpretation, one can use a provably equivalent consistency statement. 

Via a generalisation of the FGH theorem proven in \cite{Joosten:2015:TuringJumpsThroughProvability, Joosten:2019:TransfiniteTuringjumps} one can see that the consistency notion corresponding to $[1]$ provability is in a sense $\Pi^0_2$ complete and so on. Thus, it makes sense to consider the following recursion as in \cite{Joosten:2015:TuringJumpsThroughProvability}: provability at level $n$ means provable from an oracle which is a consistency statement of level $m$ for some $m<n$. It feels like lifting oneself up from the swamp by pulling ones hairs as the \emph{Baron von M\"unchhausen} did. Moreover, the recursion lends itself to an easy transfinite generalisation and that is exactly what this paper does. Before we close the introduction with an overview of how the current paper does so, we would like to point out how this paper fits in the landscape of related literature thereby trying to provide an ample justification for it.

Ordinal analysis via polymodal provability logics seems to have various benefits over other methods of ordinal analysis. An important benefit is it allows to tell different incomplete theories apart at the lowest possible level of $\Pi^0_1$ sentences. It is good to recall that the classical $\Pi^1_1$ proof theoretical ordinal will not even discern theories at the level of $\Sigma^1_1$-level. Another benefit may seem the modularity of ordinal analysis: the ordinal analysis of different theories will all share the same template and re-use various tools and theorems. 

We see another stronghold in the fact that the approach relates various different fields in a natural way. In particular, the closed formulas of \glp --called \emph{worms}-- are important in this. Worms can be used to denote various notions central to foundational issues. For one, they are simple and well-behaved elements from a well-behaved logic. Even though the logic \glp is known to be PSPACE-complete (\cite{Shapirovsky:2008:PSPACEcompletenessOfGLP}) it is Kripke incomplete. However, natural topological semantics do exist (\cite{Icard:2009:TopologyGLP, Ignatiev:1993:StrongProvabilityPredicates, BeklemishevGabelaia:2011:TopologicalCompletenessGLP, Fernandez:2012:TopologicalCompleteness, AguileraFernandez:2017:strongCompleteness}) even though it is known to depend on strong cardinal assumptions for various natural topological spaces \cite{BagariaMagidorSakai:2013}. 

Moreover, the closed fragment of $\glp_\Lambda$ is very well behaved, well studied and in particular does allow for natural relational semantics \cite{Ignatiev:1993:StrongProvabilityPredicates, FernandezJoosten:2012:KripkeSemanticsGLP0, FernandezJoosten:2013:ModelsOfGLP}. In addition, and this provides a second interpretation of worms, the worms are known to define a well-ordered relation as studied in \cite{Beklemishev:2005:VeblenInGLP, BeklemishevFernandezJoosten:2014:LinearlyOrderedGLP, FernandezJoosten:2014:WellOrders} and thus can provide for ordinal notation systems (\cite{Beklemishev:2005:VeblenInGLP, Fernandez:2017:Spiders, HermoFernandez:2019:BracketCalculus}).

Some simple worms are just consistency statements which are known to be related to reflection principles so that by classical results they are related to fragments of arithmetic \cite{KreiselLevy:1968:ReflectionPrinciplesAndTheirUse}. Thus, worms --apart from being privileged elements of a decidable logic-- can denote both ordinals and fragments of arithmetic. A possibly more important use however lies in their relation to Turing progressions: each Turing progression below $\varepsilon_0$ can be approximated by the arithmetical interpretation of a $\glp_\omega$ worm. The relation goes even that far so that points in a universal modal model for the closed fragment of $\glp_\omega$ can be seen as arithmetical theories axiomatised by Turing progressions (\cite{Joosten:2016:TuringTaylorExpansion}) so that the model displays all conservation results between the different theories. It is these four different possible denotations for worms that make them so versatile and make new interpretations of $\glp_\Lambda$ as the current paper so promising.

\paragraph{\bf Plan of the paper}
Section \ref{section:Prelims} provides some useful lemmata and settles on notation which otherwise is quite standard so that it can be skipped by the initiate readers only to come back to it when needed.  Then, in Section \ref{section:TuringJumpProvability} the central provability notion of this paper is introduced: one-M\"unchhausen provability. The usage of the word ``one" in there refers to the fact that provability at level $\alpha$ is allowed to use a single oracle sentence of a lower level consistency statement.

Section \ref{section:OnUniqueness} mainly dwells on the fact that in general we can not prove that different M\"unchhausen provability predicates are provably equivalent even if they are so on the low levels. It is observed that we do have uniqueness in case the object theory and the meta theory are provably the same.  

Section \ref{section:munchhausenSound} then proceeds to prove soundness for one-M\"unchhausen provability for a large class of theories and Section \ref{section:completeness} proves arithmetical completeness. In Section \ref{section:formalisation} it is sketched how one-M\"unchhausen provability can be formalised in second order arithmetic. The formalisation requires a substantial amount of transfinte induction both in the object and meta theory so that applications to ordinal analysis will become difficult. Finally, in Section \ref{section:MunchhausenProvabilityMultipleOracle} some first steps are taken on how to weaken the needed strength of the object and meta theory. By allowing for multiple oracles sentences instead of just one, soundness can be proven without any transfinite induction.

\section{Preliminaries}\label{section:Prelims}

In this section we dwell succinctly on the necessary notions from both formal arithmetic and modal provability logics. Apart from proving a few new observations, we mainly settle on notation and refer to the literature for details.

\subsection{Arithmetic}
This paper deals with interpretations of transfinite provability logic. Even though the set-up of such interpretations starts schematically so that our analysis applies to a wide range of theories, we will in particular have second order arithmetic in mind. We refer the reader to standard references for details (\cite{Simpson:2009:SubsystemsOfSecondOrderArithmetic, HajekPudlak:1993:Metamathematics, Beklemishev:2005:Survey}) and only include some minimal comments for expository purposes.

For first-order arithmetic, we shall work with theories with identity in the language $\{  0,1, \exp, +, \cdot, < \}$ of arithmetic where $\exp$ denotes the unary function $x\mapsto 2^x$. We define $\Delta^0_0=\Sigma^0_0=\Pi^0_0$ formulas (also referred to as \emph{elementary formulas}) as those where all quantifiers occur bounded, that is we only allow quantifiers of the form $\forall\,  x{<}t$ or $\exists\, x{<}t$ where $t$ is some term not containing $x$. We inductively define $\Pi^0_{n+1}$/$\Sigma^0_{n+1}$ formulas as allowing a block of universal/existential quantifiers up-front a $\Sigma^0_n$/$\Pi^0_n$ formula. The union of these classes is called the \emph{arithmetical formulas} and denoted by $\Pi^0_\omega$. 

If $P$ is a predicate, the classes relativized to $P$ are defined the same with the sole difference that we consider the predicate $P$ as an atomic formula. We flag relativisation by including the predicate in brackets after the class like, for example, in $\Pi^0_1(P)$.  

\emph{Peano Arithmetic} (\pa) contains the basic axioms describing the non-logical symbols together with induction formulas $I_\varphi$ for any formula $\varphi$ where as always $I_\varphi: = \varphi(0) \wedge \forall x(\varphi (x) \to \varphi (x+1)) \to \forall x \varphi (x)$. When $\Gamma$ is a complexity class, by $\mathrm{I}\Gamma$ we denote the theory which is like \pa  except that induction is restricted to formulas in $\Gamma$. The theory $\mathrm{I}\Delta^0_0$ is also referred to as \emph{elementary arithmetic\footnote{In the literature it is more common to work with a formulation of \ea in the language without exponentiation. For the purpose of this paper, the differences are not essential.}} or \emph{Kalmar elementary arithmetic} (\ea).

In this paper we also mention collection axioms $\mathrm{B}_\varphi$ which basically state that the range of a function with finite domain is finite: $B_\varphi: = \forall z{<}y\, \exists x \varphi (z,x) \to \exists u\, \forall z{<}y\, \exists\, x{<}u \varphi (z,x)$. Again, for a formula class $\Gamma$, by $\mathrm{B}\Gamma$ we denote the set of collection axioms for formulas from $\Gamma$.

Second order arithmetic is an extension of first order arithmetic where we now add second order set variables together with a binary symbol $\in$ for membership. Instead of extending identity to second order terms we stipulate that second order identity is governed by extensionality: $X = Y :\Leftrightarrow \forall x\ (x\in X \leftrightarrow x\in Y)$. The formula classes $\Sigma^1_n$ and $\Pi^1_n$ are defined as their first-order counterpart only that we now count second order quantification alternations. Likewise, by $\Pi^1_\omega$ we denote the class of all second order formulas.

The strength of various fragments of second order arithmetics is in large determined by their set existence axioms. The \emph{collection axiom} for $\varphi$ tells us that $\varphi$ defines a set: $\exists X \forall x (x\in X \leftrightarrow \varphi)$. The second order system $\AcaNaught$ contains the defining axioms for the first-order non-logical symbols together with set-induction $0\in X \wedge \forall x (x{\in} X \to x{+}1 {\in} X) \to \forall x \ x{\in} X$ and collection for all arithmetical formulas. 

The theory $\AcaNaught$ is conservative over $\pa$ for first-order formulas. In \cite{FernandezJoosten:2018:OmegaRuleInterpretationGLP} the system \eca is introduced as $\AcaNaught$ except that comprehension is restricted to $\Delta^0_0$ formulas. In \cite[Lemma 3.2]{CordonFernandezJoostenLara:2017:PredicativityThroughTransfiniteReflection} it is proven that \eca is conservative over \ea for first-order formulas.

We will tacitly assume that when we are given a theory $T$, we are actually given a decidable formula $\tau$ that binumerates the axioms of $T$. That is to say, $\chi$ is an axiom of $T$ if and only if\footnote{We shall refrain from making a difference between syntactical objects and their G\"odel numbers when the context allows us so.} $T\vdash \tau(\chi)$. For each theory $T$ we denote by $\Box_T$ the unary $\Sigma^0_1$-predicate that defines provability in $T$. That is, $\mathbb N \models \varphi$ if and only if $\varphi$ is provable in $T$. When we write $\Box_T \varphi(\dot x)$ we denote the formula with free variable $x$ that expresses that for each number x, the formula $\varphi (\overline n)$ is provable in $T$. Here, $\overline n$ denotes the numeral of $n$ which is a syntactical expression denoting $n$, for example defined as $\overline 0 = 0; \overline {x+1} =\overline x + 1$.

The Friedman-Goldfarb-Harrington Theorem (FGH for short) states that for any computably enumerable theory $U$, the corresponding formalised provability predicate is provably $\Sigma^0_1$-complete provided $U$ is consistent. Since the theorem provides an important tool in this paper, let us give a precise formulation.

\begin{theorem}[Friedman-Goldfarb-Harrington]\label{theorem:FGH}
Let $U$ be a computably enumerable theory with corresponding provability predicate $\Box_U$. We have that for any $\Sigma^0_1$ formula $\sigma(x)$, there is a $\Sigma^0_1$ formula $\rho(x)$ so that 
\[
\ea \vdash \Diamond_U\top \to \forall x\ \big( \sigma (x) \leftrightarrow \Box_U \rho (\dot x)\big).
\]
\end{theorem}

The theorem was given its name in \cite{Visser:2005:FaithAndFalsity} in acknowledgment to the intellectual parents. Generalisations to other arithmetical provability predicates were studied in \cite{Joosten:2015:TuringJumpsThroughProvability} and \cite{Joosten:2019:TransfiniteTuringjumps}. In particular, the quantification over $\Sigma^0_1$ formulas (without exponentiation however) can be made internal in \ea and the $\rho$ is obtained from $\sigma$ by means of an elementary function.

\subsection{Transfinite provability logic}

Even though via the FGH theorem the provability predicate $\Box_T$ is in a sense $\Sigma_1$ complete for a wide variety of theories, the provable structural behaviour of the predicate can be described with well-behaved PSPACE decidable propositional modal logics. 

The simplest modal logics have one unary modal operator $\Box$ which syntactically behaves like negation. The dual modality $\Diamond$ can be seen as an abbreviation of $\neg \Box \neg$. The basic logic \logic K is axiomatised by all propositional tautologies (in the signature with $\Box$) and all so-called distribution axioms $\Box (A \to B) \to (\Box A \to \Box B)$. The rules of \logic K are modus ponens and Necessitation: from $A$ conclude $\Box A$. 

The logic \logic{K4} arises by adding the transitivity axioms to \logic K: $\Box A \to \Box \Box A$. G\"odel L\"ob's logic \logic{GL} arises to adding L\"ob's axiom scheme to \logic K: $\Box (\Box A \to A) \to \Box A$. It is known that \logic{GL} is a proper extension of \logic{K4} and that it exactly describes the provable structural properties of the provability predicate for a wide range of theories.

In this paper we are interested in provability logics of a collection of provability predicates $[\alpha]$ of increasing strength indexed by ordinals $\alpha$. For the finite ordinals, this logic was discovered by Japaridze in \cite{Japaridze:1988}. We now present this logic, which would be  $\glp_\omega$ in our notation as given in the following definition.

\begin{definition} 
For $\Lambda$ an ordinal or the class of all ordinals, the logic $\mathsf{GLP}_\Lambda$ is given by the following axioms:
\begin{enumerate}
\item all propositional tautologies{,}
\item Distributivity:
$[\xi](\varphi \to \psi) \to ([\xi]\varphi \to [\xi]\psi)$ for all $\xi<\Lambda${,}

\item Transitivity:
{$[\xi] \varphi \to [\xi] [\xi]\varphi$ for all $\xi<\Lambda$}{,}

\item L\"ob:
{$[\xi]([\xi]\varphi \to \varphi)\to[\xi]\varphi$ for all $\xi<\Lambda$}{,}

\item Negative introspection:
$\<\zeta\>\varphi\to\<\xi\>\varphi$ for $\xi<\zeta<\Lambda${,}

\item Monotonicity:
$\<\xi\>\varphi\to [\zeta]\<\xi\>\varphi$ for $\xi<\zeta<\Lambda$.
\end{enumerate}
The rules are Modes Ponens and Necessitation for each modality: $\displaystyle \frac{\varphi}{\nboxBox{\xi}\varphi}$.
\end{definition}

The following lemma is proven in \cite{BeklemishevFernandezJoosten:2014:LinearlyOrderedGLP}.

\begin{lemma}\label{theorem:GLPconservativelyExtendsFragments}
The logic $\glp_\Lambda$ is conservative over $\glp_{\Lambda'}$ for $\Lambda'<\Lambda$.
\end{lemma}

The lemma is particularly useful in proofs where you only have access to reasoning up to $\glp_\Lambda'$ and tells you that any statement formulated in this fragment can actually be proven there. We shall use this result throughout the paper, mostly without explicit mention. Let us now prove some basic properties that shall be needed later in the paper.

\begin{lemma}\label{theorem:basicGLPlemmas}\ 
\begin{enumerate}
\item\label{item:consistencyProvable:theorem:basicGLPlemmas}
$\glp \vdash [\alpha] \la \beta \ra \top$ whenever $\alpha>\beta$;

\item \label{item:bigConsEquivSmallCons:theorem:basicGLPlemmas}
For $\alpha > \beta$ we have $\glp \vdash \la \alpha \ra \top \to \Big( \la \beta\ra \varphi \leftrightarrow \la \alpha \ra \la \beta \ra \varphi \Big)$.

\item\label{item:boxDisjunctionDiamond:theorem:basicGLPlemmas}
For $\alpha \geq \beta >0$ we have $\glp \vdash \la \alpha \ra \top \ \to \  \Big( \la \beta \ra \phi \vee \Box \psi \Big) \ \leftrightarrow \ \Big(  \la \beta \ra \big( \phi \vee \Box \psi \big) \Big)$.

\end{enumerate}
\end{lemma}

\begin{proof}
We reason in \glp.

For {\bf Item \ref{item:consistencyProvable:theorem:basicGLPlemmas}}:
If $\la \beta \ra \top$, then $[\alpha] \la \beta \ra \top$ by the negative introspection axiom. In case $[\beta]\bot$ we get by an \emph{ex falso} under the $[\beta]$ modality that $[\beta]\la \beta \ra \top$ whence $[\alpha]\la \beta \ra \top$ by monotonicity.

For {\bf Item \ref{item:bigConsEquivSmallCons:theorem:basicGLPlemmas}} we work under the assumption that $\la \alpha \ra \top$. From $\la \beta \ra \varphi$ we get, since $\beta <\alpha$, that $[\alpha]\la \beta\ra \varphi$ so from $\la \alpha \ra \top$ we get $\la \alpha \ra \la \beta \ra\varphi$. For the other direction, from $\la \alpha \ra \la \beta \ra \varphi$ we get by monotonicity that $\la \beta \ra \la \beta \ra \varphi$ whence by transitivity we obtain the required $\la \beta \ra \varphi$.

For {\bf Item \ref{item:boxDisjunctionDiamond:theorem:basicGLPlemmas}}: We now work under the assumption that $\la \alpha \ra \top$. The only case to consider in the $\rightarrow$ direction is when $\Box \psi$ holds. Then, $\Box \Box \psi$ whence $[\alpha] \Box \psi$ which together with $\la \alpha \ra \top$ yields $\la \alpha \ra \Box\psi$ whence $\la \beta \ra \big( \phi \vee \Box \psi \big)$.

For the $\leftarrow$ direction
we need to prove $\la \beta \ra \big( \phi \vee \Box \psi \big) \to \la \beta \ra \phi \vee \Box \psi$. So, suppose $\la \beta \ra \big( \phi \vee \Box \psi \big)$ and $\neg \Box  \psi$ whence $[\beta]\neg \Box \psi$. But since $\la \beta \ra \big( \phi \vee \Box \psi \big)$ we must have $\la \beta \ra \phi$ and by weakening $\la \beta \ra  \phi \vee \Box \psi$.
\end{proof}

\subsection{Transfinite induction and its kin}
In various arguments we will have to prove that a statement $\varphi$ holds for all ordinals $\alpha$. Often we will prove this by transfinite recursion on $\alpha$. However, in certain cases, transfinite induction is not available. In such cases there is a technique called \emph{reflexive induction}. 

The principle of reflexive induction can syntactically be seen as twice weakening regular transfinite induction. Recall that transfinite induction for a formula $\varphi$ is 
\[
{\sf TI}_\varphi \ \ := \ \ \forall \alpha \big( \forall \, \beta{<}\alpha \varphi (\beta) \to \varphi (\alpha) \big) \ \to \ \forall \alpha \varphi (\alpha).
\]
and for a set of formulas $\Gamma$ the principle ${\sf TI}(\Gamma)$ denotes the collection of all ${\sf TI}_\varphi$ for $\varphi \in \Gamma$. As a first weakening one could consider the rule based version: from $T\vdash \forall \alpha \big( \forall \, \beta{<}\alpha \varphi (\beta) \to \varphi (\alpha) \big)$, conclude $T\vdash \forall \alpha \varphi (\alpha)$. Now, one can change the antecendent to $T\vdash \forall \alpha \big( \Box_T\forall \, \beta{<}\alpha \varphi (\dot \beta) \to \varphi (\alpha) \big)$ to arrive at reflexive induction. However, it turns out that by doing so, it has lost all its strength. That, is, the resulting principle is provable in almost any theory:

\begin{theorem}[Reflexive induction]\label{theorem:ReflexiveTransfiniteInduction}
Let $T$ be any theory capable of coding syntax. If $T\vdash \forall \alpha \Big( \Box_T \big(\forall \beta < \dot \alpha \ \varphi (\beta)\big) \to \varphi(\alpha) \Big)$, then
$T\vdash \forall \alpha \varphi (\alpha)$.
\end{theorem}

Although this principle is well known since Schmerl's work (\cite{Schmerl:1978:FineStructure}) we include a proof to emphasize that the principle actually does not rely at all on the fact that $<$ is a well-order. As a matter of fact, the proof goes through for any kind of relation and basically boils down to an application of L\"ob's Theorem.

\begin{proof}
We shall see that from the assumption 
\[
T\vdash \forall \alpha \Big( \Box_T \big(\forall \, \beta {<} \dot \alpha \ \varphi (\beta)\big) \to \varphi(\alpha) \Big)
\] 
we get $T\vdash \Box_T \forall \alpha \varphi (\alpha) \to \forall \alpha \varphi (\alpha)$ so that the conclusion $T\vdash \forall \alpha \varphi (\alpha)$ follows by L\"ob's Theorem.

Thus, we reason in $T$, pick $\alpha$ arbitrary, we assume $\Box_T \forall \alpha \varphi (\alpha)$, or equivalently $\Box_T \forall \theta \varphi (\theta)$, and set out to prove $\varphi (\alpha)$. But using $\Box_T \big(\forall \, \beta {<} \dot \alpha \ \varphi (\beta)\big) \to \varphi(\alpha)$ in the last step of the following reasoning, we clearly have
\[ 
\begin{array}{lll}
\Box_T \forall \theta \varphi (\theta) &\to& \Box_T \forall \theta \forall \, \beta{<}\theta \ \varphi (\beta)\\
 \ & \to &\forall \theta \, \Box_T \forall \, \beta{<}\dot\theta \ \varphi (\beta)\\
 \ & \to &\Box_T \forall \, \beta{<}\dot\alpha \ \varphi (\beta)\\
 \ & \to &\varphi (\alpha).\\
\end{array}
\]
\end{proof}

On occasion, in this paper we will have to combine regular transfinite induction and reflexive induction. We call this amalgamate \emph{transfinite reflexive induction}.

\begin{lemma}[Transfinite reflexive induction]\label{theorem:TransfiniteReflexiveInduction}
Let $T$ be a theory with a sufficient amount of transfinite induction as specified below and let $\prec$ be a well-order in $T$. 
\[
\begin{array}{ll}
\mbox{If } & \\
 & T\vdash \forall \, \alpha \Big(\forall\, \beta{\prec}\alpha \, \varphi(\beta) \ \wedge \ \Box_T \big( \forall\, \beta{\prec}\dot\alpha \, \varphi(\beta)\big) \ \to \ \varphi (\alpha) \Big),\\
\mbox{then} & \\
 & T\vdash \forall \alpha \ \varphi (\alpha).
\end{array}
\] 
To prove transfinite reflexive induction for $\varphi$ it suffices that $T$ is capable of coding syntax and proves transfinite induction for formulas of the form $\Box_T\chi \to \varphi$. 
\end{lemma}

\begin{proof}
To start our proof we assume 

\begin{equation}\label{equation:AssumptionTRI}
T\vdash \forall \, \alpha \Big(\forall\, \beta{\prec}\alpha \, \varphi(\beta) \ \wedge \ \Box_T \big( \forall\, \beta{\prec}\dot\alpha \, \varphi(\beta)\big) \ \to \ \varphi (\alpha) \Big).
\end{equation}

We will prove by transfinite induction on $\alpha$ that 

\begin{equation}\label{equation:AssumptionRTI}
T \vdash \forall \alpha\ \Big(\Box_T \forall\, \beta{\prec}\dot\alpha \, \varphi (\beta) \to \varphi (\alpha)\Big)
\end{equation}

so that the result $T\vdash \forall \alpha \, \varphi (\alpha)$ follows by reflexive induction (Lemma \ref{theorem:ReflexiveTransfiniteInduction}). Proving \eqref{equation:AssumptionRTI} for $\alpha=0$ amounts to showing that $T\vdash \varphi(0)$ which follows directly from \eqref{equation:AssumptionTRI}.

For the inductive step, we reason in $T$, fix some $\alpha>0$, assume that 

\begin{equation}\label{equation:InductiveAssumptionRTI}
\forall \, \beta{\prec}\alpha\ \Big(\Box_T \forall\, \gamma{\prec}\dot\beta \, \varphi (\gamma) \to \varphi (\beta)\Big)
\end{equation}

and set out to prove

\begin{equation}\label{equation:InductiveStepRTI}
 \Box_T \forall\, \gamma{\prec}\dot\alpha \, \varphi (\gamma) \to \varphi (\alpha).
\end{equation}

To this end, we further assume that $\Box_T \forall\, \gamma{\prec}\dot\alpha \, \varphi (\gamma)$, so that certainly we have $\forall \, \beta{\prec}\alpha\ \Box_T \forall\, \gamma{\prec}\dot\beta \, \varphi (\gamma)$. Combining the latter with \eqref{equation:InductiveAssumptionRTI} yields $\forall \, \beta{\prec}\alpha\ \varphi(\beta)$. This, together with our assumption $\Box_T \forall\, \gamma{\prec}\dot\alpha \, \varphi (\gamma)$ is the antecedent of \eqref{equation:AssumptionTRI} so that we may conclude $\varphi (\alpha)$ which finishes the proof.
\end{proof}

\section{Theories for Single Oracle M\"unchhausen provability}\label{section:TuringJumpProvability}

Throughout this section, we fix some ordinal $\Lambda$ and understand that all ordinals denoted in this section are majorized by $\Lambda$.

\subsection{Single Oracle M\"unchhausen provability}

We are interested in theories $T$ that can formalize a provability notion so that provably in $T$ the following recursion holds

\begin{equation}\label{equation:ramifiedProvability}
{[\zeta]}^\Lambda_T \phi \ \ 
:\Leftrightarrow \ \ 
\Box_T \phi \ \vee \ \exists \psi\, \exists \, \xi{<}\zeta\ \big({\la \xi \ra}^\Lambda_T \psi \ \wedge \ \Box_T ({\la \xi \ra}^\Lambda_T \psi \to \phi)\big).
\end{equation}

Here, $\Box_T \varphi$ will denote a standard predicate on the natural numbers expressing ``the formula (with G\"odel number) $\varphi$ is provable in the theory $T$". Further, it is understood that ${\la \xi \ra}^\Lambda_T$ stands for $\neg [\xi]^\Lambda_T \neg$.

Rather than exposing a concrete theory where this recursion is formalizable in a particular way and provable, we will define a class of theories that are able to define and prove this recursion and have some additional desirable properties. 

Next we shall see which properties of the predicates $[\zeta]^\Lambda_T$ can be proven from the mere recursion defined in \eqref{equation:ramifiedProvability}. It will turn out that under some fairly general conditions we can prove the collection of predicates $[\zeta]^\Lambda_T$ for $\zeta<\Lambda$ to provide a sound interpretation for $\glp_\Lambda$. 

In Section \ref{section:completeness} we shall see that by requiring slightly more on our predicate and theory, this will give us arithmetical completeness.

In principle it would make sense to study \eqref{equation:ramifiedProvability} at a higher level of generality. For example, $T$  could be some version of set-theory allowing for uncountable $\Lambda$. As long as \eqref{equation:ramifiedProvability} is provable together withs some additional conditions, most of the results of this paper will carry over. It would be natural to require $\Box_T$ to be such that all \gl theorems are schematically provable in $T$ in such a setting.

\subsection{Theories amenable for Single Oracle M\"unchhausen provability}
For the sake of readability we shall often not distinguish between an ordinal $\alpha<\Gamma$, a notation for such an $\alpha$ or even an arithmetization of such a notation for $\alpha$. We shall however be explicit about the difference between the ordering $<$ on the ordinals and the arithmetization $\prec$ of this ordering on ordinals.

\begin{definition}
Let $T$ be a theory and let $\Lambda$ denote an ordinal equipped with a representation in the language of $T$ with corresponding represented ordering $\prec$. For this representation, it is required that
\[
\begin{array}{l}
T\vdash ``{\prec}\mbox{ is transitive, right-discrete and has a minimal element}",\\
T\vdash (\xi \prec \zeta) \to [\zeta]^\Lambda_T (\xi \prec \zeta),\\
\mbox{$\xi < \zeta < \Lambda$ implies $T\vdash \xi \prec \zeta$.\footnotemark}
\end{array}
\]
\footnotetext{This requirement can be dropped if we are happy with a soundness proof where all ordinals are internally quantified. In this case we assume that each $\alpha<\Gamma$ has a natural representation in $T$ so that it makes sense to speak about the soundness of the necessitation rule.}
We call $T$ a \emph{Single Oracle $\Lambda$-M\"unchhausen Theory} --or  a \emph{$\Lambda$-One-M\"unchhausen Theory} for short-- whenever there is a binary predicate $[\xi]_T^\Lambda \varphi$ with free variables $\xi$ and $\varphi$ so that
\[
T\vdash \forall \varphi\ \forall \alpha{\prec}\Lambda \Big(\ {[\zeta]_T^\Lambda} \phi \  
\leftrightarrow  \ 
\Box_T \phi \ \vee \ \exists \psi\, \exists \, \xi{\prec}\zeta\ \big({\la \xi \ra^\Lambda_T} \psi \ \wedge \ \Box_T (\, {\la \xi \ra^\Lambda_T} \psi \to \phi) \, \big) \, \Big).
\]
In this case, we call the binary predicate $[\xi]_T^\Lambda \varphi$ a corresponding 1-M\"unchhausen provability predicate.
\end{definition}

The ``One" in ``$\Lambda$-One-M\"unchhausen Theory" refers to the fact that provability $[\zeta]_T^\Lambda$ at level $\zeta$ makes use of one single oracle sentence $\la \xi \ra^\Lambda_T \psi$. In Section \ref{section:MunchhausenProvabilityMultipleOracle} we shall see variations where we allow various oracle sentences to occur.

Often shall we simply drop the, or some of the indices of $[\xi]^\Lambda_T$ like for example in $[\xi]_T \varphi$ in case the ordinal $\Lambda$ is clear from the context. To shorten nomenclature further, we shall mostly simply speak of 1-M\"unchhausen theories and the corresponding 1-M\"unchhausen provability. Often, when we speak of 1-M\"unchhausen theories we implicitly assume that we have fixed some 1-M\"unchhausen provability predicate $[\alpha]\varphi$.

\begin{observation}
Since any 1-M\"unchhausen theory $T$ proves that there is a $\prec$-minimal element, we shall use the notation $0$ for this element even if the natural number (or object) representing this minimal element is not the natural number zero. Likewise, from right-discreteness we know that for any element $\alpha\prec \Lambda$, there is a next bigger element that we shall suggestively call $\alpha+1$. In analogy, we shall denote $0+1$ by $1$, $1+1$ by $2$, $2+1$ by $3$, etcetera. 
\end{observation}

The following observation is immediate.
\begin{lemma}\label{theorem:ZeroMunchhausenIsNormalProvability}
Let $T$ be a Single Oracle $\Lambda$-M\"unchhausen Theory with corresponding 1-M\"unchhausen provability predicate $[\alpha]^\Lambda_T$. We have that
\[
T \vdash \ \forall \varphi \ \big( [0]_T^\Lambda \varphi  \ \leftrightarrow \Box_T \varphi \big).
\]
\end{lemma}

When working with sound theories, we know that all the corresponding 1-M\"unchhausen consistency statements are actually true:

\begin{proposition}\label{theorem:MHconsistencyStatementsAreTrue}
Let $T$ be a sound Single Oracle $\Lambda$-M\"unchhausen Theory with corresponding 1-M\"unchhausen provability predicate $[\alpha]^\Lambda_T$. Then for each $\xi \prec \Lambda$ we have $\mathbb N \models \la{\xi}\ra_T^\Lambda\top$.
\end{proposition}

\begin{proof}
By a simple case distinction. In case $\xi=0$, we get from a hypothetical $\mathbb N \models [0]_T \bot$ together with the soundness and the above lemma that $\mathbb N \models \Box_T \bot$ so that $T\vdash \bot$ which cannot be. 

In case $\xi \succ 0$, suppose for a contradiction that $\mathbb N\models [{\xi}]_T \bot$. Then, using soundness of $T$, we only need to consider the case that
\[
\mathbb N \models \exists \psi\, \exists \, \zeta{\prec}\xi\ \big(\la{\zeta}\ra_T \psi \wedge \Box_T(\la{\zeta}\ra_T \psi\to \bot) \big),
\]
so that for some ordinal $\zeta \prec \xi$ and some formula $\psi$ we have $\mathbb N \models \la{\zeta}\ra_T \psi$. Also $\mathbb N \models \Box_T(\la{\zeta}\ra_T \psi\to \bot)$ so that $T\vdash \la{\zeta}\ra_T \psi\to \bot$ whence by soundness of $T$ we see that $\mathbb N \models\neg \la{\zeta}\ra_T \psi$ which is a contradiction.
\end{proof}

We note that the above argument does not use transfinite induction. 

\section{On uniqueness of M\"unchhausen provability}\label{section:OnUniqueness}

The definition of 1-M\"unchhausen provability allows for various different 1-M\"unchhausen predicates to exist. Of course, it would be highly desirable that the defining equivalence \eqref{equation:ramifiedProvability} for 1-M\"unchhausen provability defined a $T$ provably unique predicate. We can prove uniqueness of the predicate via an external induction up to any level below $\omega$.

\begin{lemma}\label{theorem:UniquePredicatesTillOmega}
Let $T$ be a sound Single Oracle $\Lambda$-M\"unchhausen Theory for $\Lambda\geq \omega$, with corresponding 1-M\"unchhausen provability predicates $[\alpha]^\Lambda_T$ and $\overline{[\alpha]}^\Lambda_T$. We have for any natural number $n$ that
\[
T\vdash \forall \varphi \ \big( [n]^\Lambda_T \varphi  \ \leftrightarrow \ \overline{[n]}^\Lambda_T \varphi \big).
\]
\end{lemma}

\begin{proof}
We proceed by an external induction where the base case follow directly from Lemma \ref{theorem:ZeroMunchhausenIsNormalProvability}. We shall omit super and sub indices. 

For the inductive step, we reason in $T$, fix some formula $\varphi$, fix the $(n+1)$th element in the $\prec$ ordering and assume $[n+1]\varphi$. In the non-trivial case, there is some formula $\psi$ and an element $\tilde m{\prec}n+1$ so that $\la \tilde m \ra \psi$ and $\Box (\la \tilde m \ra \psi \to \varphi)$. Here we end our reasoning inside $T$.  Since we can prove that any element $\prec$-below the externally given $n+1$ is either the zero-th, or the first, or \ldots or, the $n$-th element, we know that $\tilde m$ corresponds to some natural number $m<n+1$. Thus, we can appeal to the external induction hypothesis that tells us that 
\begin{equation}\label{equation:theorem:UniquePredicatesTillOmega:IH}
T\vdash \forall \psi \ ([m] \psi \leftrightarrow \overline{[m]}\psi) 
\end{equation}
and consequently
\begin{equation}\label{equation:theorem:UniquePredicatesTillOmega:IHBoxed}
T\vdash \Box \forall \psi \ ([m] \psi \leftrightarrow \overline{[m]}\psi).
\end{equation}
These two ingredients are sufficient to conclude $\overline{[m]} \psi$. Of course the other direction goes exactly the same.
\end{proof}

Let us make some observations about this simple proof. First, we observe that we could only conclude \eqref{equation:theorem:UniquePredicatesTillOmega:IHBoxed} from \eqref{equation:theorem:UniquePredicatesTillOmega:IH} by necessitation since the meta-theory as in $T\vdash \ldots$ is the same as the object-theory as in $\Box_T$. Second, we observe that we only had access to the inductive hypothesis since we can express in the language of first order logic that being smaller than the $(n+1)$th element implies being equal to one of the zero-th, or \ldots, or the $n$th element. Of course, we cannot generalize this to the first limit ordinal and hence our external induction cannot be extended to the transfinite. 

If we wish to generalize our argument to the transfinite, we should replace our external induction by an internal one. Of course, then in our meta-theory, we should have access to transfinite induction. However, we only see how to continue the proof in the case where the object theory equals the meta-theory and consequently also has the same amount of transfinite induciton.

\begin{lemma}\label{theorem:UniqueExtensionalProgressions}
Let $T$ be a theory that proves the recursion from \eqref{equation:ramifiedProvability} for two predicates $[\zeta]_U$ and $\overline{[\zeta]}_U$. We further suppose that $T$ proves the basic facts about the ordering $\la \Lambda,\prec \ra$. Also, we assume that $T$ proves transfinite $\Pi_1([\alpha], \overline{[\alpha]})$ induction.

If $T$ and $U$ are $T$-provably equivalent, then we have that $[\zeta]_U$ and $\overline{[\zeta]}_U$ are $T$-provably equivalent predicates. 
\end{lemma}

\begin{proof}
We have chosen a formulation where $T$ and $U$ are different from the outset so that we clearly see at what point we need to assume that $T$ is $T$-provably equivalent to $U$. 

Thus, we reason in $T$ and will as a first attempt prove by transfinite $\Pi_1([\alpha], \overline{[\alpha]})$ induction that 
\[
\forall \zeta \,\forall  \varphi \ ([\zeta]_U\varphi \ \leftrightarrow \ \overline{[\zeta]}_U \varphi).
\]

For $\zeta=0$ the equivalence is obvious. Thus, we fix some $\zeta \succ 0$ and focus on one implication the other being analogous. Thus, we assume that $[\zeta]_U \varphi$ and set out to prove $\overline{[\zeta]}_U \varphi$. 

From the assumption $[\zeta]_U \varphi$ we find --in the non-trivial case-- some formula $\psi$ and ordinal $\xi\prec\zeta$ so that $\la \xi \ra_U \psi$ and $\Box_U (\la \xi \ra_U  \psi \to \varphi)$. The inductive hypothesis now will tell us that $\la \xi \ra_U \psi \leftrightarrow \overline{\la \xi \ra}_U \psi$. 

However, there is no way that we know that this equivalence is \emph{provable}, that is, that we have $\Box_U \Big( \la \xi \ra_U \psi \leftrightarrow \overline{\la \xi \ra}_U \psi\Big)$. The latter would be needed to conclude $\Box_U (\overline{\la \xi \ra}_U  \psi \to \varphi)$ so that $\overline{[\zeta]}_U \varphi$.

The problem cannot be solved by strengthening the induction to for example
\[
\forall \phi\  \big [\  ([\zeta]_U \phi \leftrightarrow \overline{[\zeta]}_U\phi) \ \wedge\ \Box_U ([\dot\zeta]_U \phi \leftrightarrow \overline{[\dot \zeta]}_U\phi)   \  \big ]
\]
since then the problem will simply come back but now under a box.

However, when $T = U$ we have access to transfinite reflexive induction as formulated in Lemma \ref{theorem:TransfiniteReflexiveInduction}. That is, in order to show that $\forall  \varphi \ ([\zeta]_U\varphi \ \leftrightarrow \ \overline{[\zeta]}_U \varphi)$ for a particular $\zeta$ we may assume both $\forall \, \xi{\prec}\zeta\, \forall  \varphi \ ([\zeta]_U\varphi \ \leftrightarrow \ \overline{[\zeta]}_U \varphi)$ and also $\Box_U \big (\forall \, \xi{\prec}\dot \zeta\, \forall  \varphi \ ([\zeta]_U\varphi \ \leftrightarrow \ \overline{[\zeta]}_U \varphi) \big)$ which makes that the proof now goes through easily.

\end{proof}

This lemma tells us that solutions to the recursion equivalence \eqref{equation:ramifiedProvability} need not be provably unique if the object theory $U$ is different from the meta theory $T$ or in case we do not have the sufficient amount of transfinite induction available. Not having provably unique fixpoints need not necessarily be a big problem and similar phenomena occur with for example Rosser fixpoints. 

However, as we shall see in Section \ref{section:munchhausenSound}, we also need the object theory to be equal to the meta theory if we wish to prove the soundness of $\glp_\Lambda$ with respect to the $[\zeta]^\Lambda_U$ predicates. In particular, the arithmetical soundness of the Necessitation rule requires the object and meta theory to be equal.

In case the object theory is not equal to the meta-theory, we can only prove a weak form of uniqueness as expressed in the following lemma.

\begin{lemma}\label{theorem:weakUniqueExtensionalProgressions}
Let $T$ be a theory that proves the recursion expressed in equation \eqref{equation:ramifiedProvability} for two predicates $[\zeta]_U$ and $\overline{[\zeta]}_{V}$ with $V$ possibly different from $U$. We further suppose that $T$ proves the basic facts about the ordering $\la \Lambda,\prec \ra$. Also, we assume that $T$ proves transfinite $\Pi_1([\alpha]_U, \overline{[\alpha]_V})$ induction.
In case $T$ proves the arithmetical soundness of $\glp_\Lambda$ for both predicates $[\zeta]_U$ and $\overline{[\zeta]}_{V}$, then (omitting subscripts)
\[
T \vdash \forall \, \alpha {\prec}\Lambda \Big (  \ \big( \la \alpha \ra \top  \leftrightarrow  \overline {\la \alpha\ra} \top  \big) \ \longrightarrow \forall \varphi\, \exists \psi \ ([\alpha]\varphi \leftrightarrow \overline{[\alpha]} \psi)\Big).
\]
\end{lemma}

\begin{proof}
We reason in $T$ and proceed by a transfinite induction on $\alpha$. Thus, we assume the equi-consistency of both theories, fix some formula $\varphi$ and assume $[\alpha] \varphi$.  The case where $\overline{[\alpha]}\bot$ is trivial, so we assume $\overline{\la \alpha\ra} \top$ whence also ${\la \alpha\ra} \top$. Thus, in case $\neg [\alpha]\varphi$ we may by consistency use $\psi = \bot$.

In case that $[\alpha]\varphi$ in virtue of $\Box_U \varphi$, we are done by the FGH theorem (Theorem \ref{theorem:FGH}) for the theory $V$ since $\Box_U \varphi \in \Sigma^0_1$. In the other case, there are $\beta \prec \alpha$ and $\chi$ so that $\la \beta\ra \chi$ and $\Box_U (\la \beta\ra \chi \to \varphi)$. By the IH we find some $\chi'$ so that $\overline{\la \beta\ra} \chi' \ \leftrightarrow \ \la \beta\ra \chi$. Since we work under the assumption of $\overline {\la \alpha \ra}\top$ and since the provability predicates are sound for $\glp_\Lambda$ we also have $\overline {\la \alpha \ra} \,\overline{\la \beta\ra} \chi'$ by Lemma \ref{theorem:basicGLPlemmas}.\ref{item:bigConsEquivSmallCons:theorem:basicGLPlemmas} and, in particular $\Diamond_V \overline{\la \beta\ra} \chi'$. Since we now know the consistency of the theory $V + \overline{\la \beta\ra} \chi'$ we may apply the FGH theorem to obtain a $\psi$ with 
\[
\Box_U (\la \beta\ra \chi \to \varphi) \ \leftrightarrow \ \Box_{V + \overline{\la \beta\ra} \chi'} \psi.
\]
By the formalised deduction theorem we may conclude $\Box_{V} \big( \overline{\la \beta\ra} \chi' \to  \psi \big)$ whence $[\alpha]\varphi \leftrightarrow [\alpha]\psi$.
\end{proof}

In this section we have shown that in general we cannot prove that 1-M\"unchhausen provability predicates are uniquely defined by the recursion in \eqref{equation:ramifiedProvability}. Only in the finite ordinals can we prove uniqueness. This allows us to relate the provability notions from this paper to similar ones from the literature. The most prominent example is given by the predicate
\[
[n]^{\sf True}_T \varphi \ \ \ \mbox{ which stands for } \ \ \ \ \exists \, \pi {\in} \Pi^0_1\ \Big({\sf True}_{\Pi^0_1}(\pi) \wedge \Box_T\big( \pi \to \varphi \big)\Big ).
\] 
Furthermore, in \cite{Joosten:2015:TuringJumpsThroughProvability} a reading is given where the modal operators $[n]\varphi$ are interpreted as follows.

\begin{equation}\label{equation:boxBoxDefinition}
\begin{split}
\boxBox 0_T\phi \ \ \ &:= \Box_T \phi, \ \ \mbox{ and }\\
\boxBox{n+1}_T \phi  \ \ \        &:= \ \Box_T \phi \ \vee \ \exists \, \psi \  \bigvee_{0\leq m \leq n} \Big(\boxDiamond{m}_T \psi \ \wedge \ \Box (\boxDiamond{m}_T \psi \to \phi)\Big).
\end{split}
\end{equation}

Soundness for this interpretation in \pa was proven and a strong relation was given to the truth provability predicates $[n]^{\sf True}_T$. 
The next Lemma is a strengthening on the one hand since we weaken the base theory to \ea and a weakening on the other hand since we only consider two modalities. 

\begin{lemma}\label{theorem:OldPaperIsGood}
Let $T$ be a theory that contains \ea.
We have that
\begin{enumerate}
\item
$\ea \vdash \forall \varphi \ (\boxBox 1_T \varphi \leftrightarrow [1]^{\sf True}_T \varphi)$;

\item
$\glp_2$ is sound for $T$ when interpreting $[0]$ as $\Box_T$ and $[1]$ as $\boxBox 1_T$;
 
\item In case that moreover $T$ proves the $\Sigma^0_1$-collection principle we have\\
$T \vdash \forall \, \varphi\, \forall \psi\, \exists \chi \ \big( \boxBox{1}_T \varphi \vee \boxBox{1}_T \psi \leftrightarrow \boxBox{1}_T \chi\big)$.
\end{enumerate}
\end{lemma}

\begin{proof}
It is easy to prove inside \ea that $\boxBox 0 \varphi \leftrightarrow \Box \varphi$ (we omit the subscripts). Likewise, $\boxBox 0 \varphi \to \boxBox 1 \varphi$ and $\boxBox 1 \bot \to \boxBox 1 \varphi$ are easy to prove. With these ingredients the first item easily follows: one direction is obvious since any oracle sentence of the form $\Diamond \psi$ is in $\Pi^0_1$. The other direction is immediate in case $ \boxBox 0 \bot$ and in the case $\boxDiamond 1 \top$ it follows from the FGH theorem since under the consistency assumption, any $\Pi^0_1$ formula is equivalent and provably so to a formula of the form $\Diamond \psi$.

The second item follows from the first since the statement holds for the $[n]^{\sf True}_T$ provability predicates (see e.g. \cite{Beklemishev:2005:Survey}).
 
The third item is implicit in \cite{Joosten:2015:TuringJumpsThroughProvability} and explicitly stated and proven in \cite{Joosten:2019:TransfiniteTuringjumps} for the $[1]^{\sf True}_T$ predicate which suffices by the first item of this lemma.
\end{proof}

Via an easy external induction we can prove that \eqref{equation:boxBoxDefinition} and \eqref{equation:ramifiedProvability} define provably equivalent predices for all natural numbers. That is to say, if $T$ is a 1-M\"unchhausen theory, then for each natural number $n$ we have that 
\begin{equation}\label{equation:OldPaperIsSpecialCaseOfCurrentPaper}
T\vdash \forall \varphi \ \big( \boxBox{n}_T \varphi \leftrightarrow [n]_T^\Lambda \varphi \big)
\end{equation}
for any 1-M\"unchhausen provability predicate $[\alpha]_T^\Lambda$. Moreover, in Lemma \ref{theorem:UniquePredicatesTillOmega} we know that any 1-M\"unchhausen provability predicate $[\alpha]_T^\Lambda$ will be uniquely defined up to $\omega$. For later in the paper, we formulate the following corollary:

\begin{corollary}\label{theorem:FirstTwoPredicatesAreGood}
Let $T$ be a $\Lambda$-1-M\"unchhausen theory with $\Lambda>2$ and corresponding 1-M\"unchhausen provability predicate $[\alpha]_T^\Lambda$. Moreover, let $T$ contain ${\sf B}\Sigma_1^0$.
\begin{enumerate}
\item
$\glp_2$ is sound for $T$ when interpreting $[0]$ as $[0]_T^\Lambda$ and $[1]$ as $[1]^\Lambda_T$;
 
\item
$T \vdash \forall \, \varphi\, \forall \psi\, \exists \chi \ \big( [1]^\Lambda_T \varphi \vee [1]^\Lambda_T \psi \leftrightarrow [1]^\Lambda_T \chi\big)$.
\end{enumerate}
\end{corollary}

\begin{proof}
This follows directly from \eqref{equation:OldPaperIsSpecialCaseOfCurrentPaper} and Lemma \ref{theorem:OldPaperIsGood}.
\end{proof}

\section{Arithmetical Soundness for One-M\"unchhausen provability}\label{section:munchhausenSound}

In this section we will consider $\Lambda$-One-M\"unchhausen theories $T$ and their corresponding $\Lambda$-One-M\"unchhausen provability predicates for some fixed ordinal $\Lambda$ represented in $T$ . We shall see that from the mere defining recursion on the provability predicate we can obtain soundness of $\glp_\Lambda$. 

Many arguments in this section require transfinite induction. As we have observed in Subsection \ref{section:OrdinalAnalysis} this means that the base theory should also prove a decent amount of transfinite induction. In Section \ref{section:MunchhausenProvabilityMultipleOracle} we shall see how the need of transfinite induction can be circumvented by slightly altering the defining recursion.

Let us start the soundness proof by some basic observations that need very little arithmetical strength to be proven. In particular, the following facts do not require transfinite induction.

\begin{lemma}\label{theorem:boxBoxExfalso}
Let $T$ be a $\Lambda$-One-M\"unchhausen theory with corresponding provability predicate $[\xi]_T^\Lambda$. We have the following.
\begin{enumerate}
\item
$T \vdash \forall \xi\, \forall \chi\ \big ( \bBox{\xi}_T \bot \to \bBox{\xi}_T \chi \big )$ and more in general,

\item
$T \vdash \forall \xi\, \forall \varphi, \chi \ \Big( \bBox{\xi}_T \varphi \, \wedge\, \Box_T (\varphi \to \chi) \to \bBox{\xi}_T\chi\Big)$,

\item \label{item:BoxConjunctions:theorem:boxBoxExfalso}
$T \vdash \forall \varphi\, \forall \psi \ \Big( \bBox{\xi}_T \varphi \, \wedge\, \Box_T \psi \to \bBox{\xi}_T(\varphi \wedge \psi)\Big)$,

\item \label{item:existsGoesInTheBox:theorem:boxBoxExfalso}
$T \vdash \exists x \ \bBox{\xi}_T \varphi (\dot x) \ \to \ \bBox{\xi}_T \exists x\varphi ( x)$.
\end{enumerate}
\end{lemma}

\begin{proof}
Clearly, the first item follows from the second, so we reason in $T$ and  assume $\bBox{\xi}_T \varphi$. Thus, in the non-trivial case, for some $\psi$ and for some $\zeta \prec \xi$ we have $\bDiamond{\zeta}_T \psi$ and $\Box_T (\bDiamond{\zeta}_T \psi \to \varphi)$. Clearly, since $\Box_T (\varphi \to \chi)$, we have also $\Box_T (\bDiamond{\zeta}_T \psi \to \chi)$ so that $\bBox{\xi}_T\chi$. 

The third item follows from the second since in case of $\Box_T \psi$ we also have $\Box_T \big( \varphi \to (\varphi \wedge \psi)\big)$.

The fourth item follows by an easy case distinction on $\xi$ being zero or not and both cases essentially follow from the fact that provably $\exists x \Box_T \varphi (\dot x) \ \to \ \Box_T \exists x \varphi (x)$.
\end{proof}

From our defining recursion \eqref{equation:ramifiedProvability}, we get the axiom of negative introspection and the axiom of monotonicity almost for free.

\begin{lemma}
Let $\xi < \zeta < \Lambda$ be ordinals in a $\Lambda$-One-M\"unchhausen theory $T$. We have 
\begin{enumerate}
\item\label{item:negativeIntrospection:theorem:crossAxiomsSoundNonFormalized}
$T \vdash \forall \varphi\ \big( \bDiamond{\xi}_T \varphi  \ \to \  \bBox{\zeta}_T \bDiamond{\xi}_T \varphi \big)$;

\item\label{item:monotonicity:theorem:crossAxiomsSoundNonFormalized}
$T \vdash \forall \varphi \ \big ( \bBox{\xi}_T \varphi \ \to \ \bBox{\zeta}_T \varphi \big )$;

\end{enumerate}
\end{lemma}

\begin{proof}
Item \ref{item:negativeIntrospection:theorem:crossAxiomsSoundNonFormalized} is immediate since $\Box_T(\boxDiamond{\xi}_T \varphi  \to \boxDiamond{\xi}_T \varphi )$ using the fact that $\xi < \zeta$ implies $T\vdash \xi \prec \zeta$. Likewise, Item \ref{item:monotonicity:theorem:crossAxiomsSoundNonFormalized} follows directly from the definition since provably $\eta \prec \xi \to \eta \prec \zeta$ (recall that we required that M\"unchhausen theories prove the transitivity of $\prec$ and moreover, $\xi < \zeta$ implies $T\vdash \xi \prec \zeta$).
\end{proof}

It is easy yet important to observe that we actually have a formalized version of the previous lemma where we internally quantify over the ordinals. As such, the formalized lemma can be used for example in an induction where possibly non-standard ordinals are called upon. 

\begin{lemma}\label{theorem:crossAxiomsSound}
Let  $T$ be a $\Lambda$-One-M\"unchhausen theory. We have 
\begin{enumerate}
\item\label{item:negativeIntrospection:theorem:crossAxiomsSound}
$T \vdash \forall \xi {\prec} \zeta {\prec} \Lambda\, \forall \varphi\ \big( \bDiamond{\xi}_T \varphi  \ \to \  \bBox{\zeta}_T \bDiamond{\xi}_T \varphi \big)$;

\item\label{item:monotonicity:theorem:crossAxiomsSound}
$T \vdash \forall \xi {\prec} \zeta {\prec} \Lambda\,  \forall \varphi \ \big ( \bBox{\xi}_T \varphi \ \to \ \bBox{\zeta}_T \varphi \big )$;

\end{enumerate}
\end{lemma}

These cross axioms are for many interpretations of $\glp_\Lambda$ actually the harder axioms to prove sound. But in the M\"unchhausen interpretations they come almost for free.

The above lemma can also be interpreted that any 1-M\"unchhausen provability predicate is monotone in the ordinal parameter. We note that it is not trivial to see that the 1-M\"unchhausen provability predicate is monotone in the underlying base theory: Suppose that, for example we have a formulation of elementary arithmetic and axiomatic set theory so that provably $\ea \subset \zfc$. This means that for any formula $\varphi$ we have $\Box_\ea \varphi \to \Box_\zfc \varphi$. Is it now easy to see that we also have the expected $[1]_\ea \varphi \to [1]_\zfc \varphi$?

Let us suppose that $\bBox{1}_\ea \varphi$ because of some $\bDiamond{0}_\ea\psi$ with $\Box_\ea (\bDiamond{0}_\ea \psi \to \varphi)$. A priori it is not at all clear how this information will yield us a $\psi'$ so that $\Box_\zfc \Big(\bDiamond{0}_\zfc \psi' \to \varphi \Big)$ and furthermore $\bDiamond{0}_\zfc \psi'$: where would we get so much $\zfc$ consistency strength from?\footnote{We have that \zfc is much stronger than \ea, whence provably $\Diamond_\ea \chi \to \Box_\zfc \Diamond_\ea \chi$. Consequently, in this particular example we could take $\psi' = \Diamond_\ea \psi$: in case $\Box_\zfc \bot$ we trivially have $\Box_\zfc \varphi$ and $\Diamond_\zfc \top \to (\Diamond_\ea \psi \leftrightarrow \Diamond_\zfc \Diamond_\ea \psi)$. However, for general $T\subset U$ we cannot use the same formula $\Diamond_T \psi$ to guarantee $[1]_T \varphi \to [1]_U \varphi$.} 


At this point we can prove the soundness of the necessitation rule.
\begin{lemma}\label{theorem:necessitationSoundForOneMunchhausenProvability}
Let  $T$ be a $\Lambda$-One-M\"unchhausen theory with corresponding 1-M\"unchhausen provability predicate $[\alpha]_T^\Lambda$. For any $\alpha\prec\Lambda$ we have that if $T\vdash \varphi$, then $T\vdash [\alpha]^\Lambda_T \varphi$.
\end{lemma}

\begin{proof}
We will only show $\frac{\varphi}{\Box_T \varphi}$. This is sufficient since necessitation for larger ordinals $\frac{\varphi}{[\alpha]_T \varphi}$ follows from the monotonicity of the predicate in $\alpha$. But, as always $T\vdash \varphi$ can be expressed as a $\Sigma^0_1$ sentence which is true whence by $\Sigma^0_1$ completeness we get $T\vdash \Box_T \varphi$.
\end{proof}







We shall now prove the remaining \glp axioms to be sound. The following lemma which was proven in \cite{FernandezJoosten:2018:OmegaRuleInterpretationGLP}, tells us that we don't need to care about L\"ob's axiom $\boxBox{\xi} (\boxBox{\xi} \varphi \to \varphi) \to \boxBox{\xi} \varphi$. 

\begin{lemma}\label{theorem:noLoebNeeded}
Let ${\sf GL}^\blacksquare$ denote the extension of $\sf GL$ with a new operator $\blacksquare$ and the following axioms for all formulas $\phi,\mbox{ and }\psi$:
\begin{enumerate}
\item $\vdash\nc\phi\to \blacksquare\phi$,
\item $\vdash \blacksquare(\phi\to\psi)\to(\blacksquare\phi\to \blacksquare\psi)$ and,
\item $\vdash \blacksquare\phi\to \blacksquare\blacksquare\phi$.
\end{enumerate}

Then, for all $\phi$,
\[{\sf GL}^\blacksquare\vdash \blacksquare(\blacksquare\phi\to\phi)\to \blacksquare\phi.\]
\end{lemma}

%

Consequently, we only need to focus on the transitivity axioms $\bBox{\xi}\varphi \to \bBox{\xi}\bBox{\xi}\varphi$ and distribution axioms $\bBox{\xi} (\varphi \to \psi) \to (\bBox{\xi} \varphi \to \bBox{\xi} \psi)$ in our soundness proof. It is in this part where we need to assume that the object and meta theory are equal so that we have access to transfinite reflexive induction as formulated in Lemma \ref{theorem:TransfiniteReflexiveInduction}.

\begin{theorem}\label{theorem:boxGLPSound}
Let $T$ be a $\Lambda$-One-M\"unchhausen theory and let $[\alpha]^\Lambda_T$ be a corresponding provability predicate. If $T$ proves transfinite $\Pi_2^0([\alpha]^\Lambda_T)$ induction we have that
\begin{enumerate}

\item \label{item:GLPsound:theorem:boxGLPSound}
$T$ proves that all the rules and axioms of \glp are sound wr.t.~$T$ by interpreting $[\alpha]$ as $\bBox{\alpha}_T^\Lambda$; in particular

\item \label{item:distributivity:theorem:boxGLPSound}
Distributivity:
$T \vdash \forall \alpha \, \forall \varphi \, \forall \psi \ \Big( \bBox{\alpha}_T^\Lambda(\varphi \to \psi) \to (\bBox{\alpha}_T^{\Lambda}\varphi \to \bBox{\alpha}_T^{\Lambda}\psi)\Big)$;

\item \label{item:conjunctions:theorem:boxGLPSound}
Closure under conjunctions:
\[
T\vdash \forall \alpha \, \forall \varphi\, \forall \psi \ \Big(  \bBox{\alpha}_T^{\Lambda}\varphi \wedge \bBox{\alpha}_T^{\Lambda}\psi \ \ \leftrightarrow \ \ \bBox{\alpha}_T^{\Lambda}(\varphi \wedge \psi)  \Big);
\]

\item \label{item:disjunctions:theorem:boxGLPSound}
Weak closure under disjunctions:
$T\vdash \forall \alpha \, \forall \varphi\, \forall \psi \, \exists \chi \ \Big(\bBox{\alpha}_T^{\Lambda}\varphi \vee \bBox{\alpha}_T^{\Lambda}\psi \ \ \leftrightarrow \bBox{\alpha}_T^{\Lambda}\chi \Big)$;

\item \label{item:transitivity:theorem:boxGLPSound}
Transitivity:
$T \vdash \forall \alpha \, \forall \varphi\ \Big ( \bBox{\alpha}_T^{\Lambda} \varphi \to \bBox{\alpha}_T^{\Lambda} \bBox{\alpha}_T^{\Lambda}\varphi \Big )$.

\end{enumerate}
\end{theorem}

\begin{proof}
If we wish to prove Item \ref{item:GLPsound:theorem:boxGLPSound}, we should prove the soundness of the rules and of the axioms.

As to the rules, the only rules of \glp are modus ponens and a necessitation rule for each modality: $\displaystyle \frac{\varphi}{\boxBox{\xi}_T\varphi}$. As pointed out in Lemma \ref{theorem:necessitationSoundForOneMunchhausenProvability} the soundness of the necessitation rules follows from necessitation for $\Box_T$ and by monotonicity, Lemma \ref{theorem:crossAxiomsSound}. As always, the soundness of modus ponens is immediate.

In the remainder of our proof we shall thus focus on the axioms. Since we proved the correctness of the negative introspection axioms -- axioms of the form $\la \beta \ra \varphi \to \bBox{\alpha} \la \beta \ra \varphi$ for $\beta < \alpha$--and of the monotonicity axioms --axioms of the form $\bBox{\beta} \varphi \to \bBox{\alpha} \varphi$ for $\beta < \alpha$-- without any induction in Lemma \ref{theorem:crossAxiomsSound} and since by Lemma \ref{theorem:noLoebNeeded} we may disregard L\"ob's axiom, we set out to prove the remaining axioms which are just the distribution and the transitivity axioms to complete a proof of Item \ref{item:GLPsound:theorem:boxGLPSound}. In other words, to complete the proof of Item \ref{item:GLPsound:theorem:boxGLPSound} we should prove Items \ref{item:distributivity:theorem:boxGLPSound} and \ref{item:transitivity:theorem:boxGLPSound}.

To prove that both items hold up to a certain level $\alpha<\Lambda$ we proceed by an internal transfinite reflexive induction on $\alpha$ as expressed in Lemma \ref{theorem:TransfiniteReflexiveInduction}. We need to prove both items simultaneously since they depend on each other. As a matter of fact, to get the proof going we will need to do some induction building and prove Items \ref{item:distributivity:theorem:boxGLPSound} -- \ref{item:transitivity:theorem:boxGLPSound} of the proof simultaneously by a transfinite reflexive induction on $\alpha$.

Thus, we will reason in $T$ and shall mostly omit the subscript $T$ and superscript $\Lambda$ in the remainder of this proof. The base case of the theorem is known to hold via the soundness of \gl and the FGH theorem. 

For the reflexive inductive step, we are to prove our four items (Items \ref{item:distributivity:theorem:boxGLPSound} -- \ref{item:transitivity:theorem:boxGLPSound}) at level $\alpha$ assuming that 
we have access to all four items at any level $\beta\prec\alpha$ and we also have these four items under a regular provability predicate $\Box_T$ at any level $\beta'\prec \alpha$. As we observed before, Item \ref{item:GLPsound:theorem:boxGLPSound} at level $\alpha$ (soundness of $\glp_\alpha$) follows directly from Items \ref{item:distributivity:theorem:boxGLPSound} -- \ref{item:transitivity:theorem:boxGLPSound} for levels $\beta\prec\alpha$. Thus, we may in our inductive step assume that we have access --and $T$-provably so-- to all $\glp_\alpha$ reasoning. Let us thus focus on the first item to prove:

{\bf Item \ref{item:conjunctions:theorem:boxGLPSound}}:
$ \forall \varphi\, \forall \psi \ \Big(  \bBox{\alpha}_T^{\Lambda}\varphi \wedge \bBox{\alpha}_T^{\Lambda}\psi \ \ \leftrightarrow \ \ \bBox{\alpha}_T^{\Lambda}(\varphi \wedge \psi)  \Big)$. We fix some $\varphi$ and $\psi$ and assume $\bBox{\alpha}\varphi$ and $\bBox{\alpha}\psi$. We consider two cases. In the easy case, we have that at least one of $\Box\varphi$ or $\Box \psi$ holds in which case the result directly follows from Lemma \ref{theorem:boxBoxExfalso}.\ref{item:BoxConjunctions:theorem:boxBoxExfalso}.
  
In the remaining case, by the recursion equation for $\bBox{\alpha}$, we find ordinals $\beta, \beta' <\alpha$ and some formulas $\varphi', \psi'$ so that $\bDiamond{\beta}\varphi'$, $\bDiamond{\beta'}\psi'$, $\Box \big(\bDiamond{\beta}\varphi' \to \varphi\big)$ and $\Box \big(\bDiamond{\beta'}\psi'\to \psi \big)$.

We first remark that w.l.o.g.~we may assume $\beta'=\beta$. For, if e.g.~$\beta'<\beta$, then by Lemma \ref{theorem:basicGLPlemmas}.\ref{item:bigConsEquivSmallCons:theorem:basicGLPlemmas} we see that $\bDiamond{\beta}\top \ \to \ \big ( \bDiamond{\beta'}\psi' \ \leftrightarrow \ \bDiamond{\beta}\bDiamond{\beta'}\psi'\big)$ with $\bDiamond{\beta}\varphi' \to \bDiamond{\beta}\top$. Since we perform a transfinite \emph{reflexive} induction, we also have our inductive hypotheses under a $\Box$ and in particular $\Box\big(\bDiamond{\beta}\bDiamond{\beta'}\psi' \to \bDiamond{\beta'}\psi' \big)$. Thus, we see that $\bDiamond{\beta}\bDiamond{\beta'}\psi' \wedge \Box\big(\bDiamond{\beta}\bDiamond{\beta'}\psi' \to \psi \big)$ whence 
\[
\exists \psi'' \ \Big(\bDiamond{\beta}\psi'' \wedge \Box\big(\bDiamond{\beta}\psi'' \to \psi \big) \Big).
\]

So, we assume $\beta'=\beta < \alpha$, and by the inductive hypothesis (on Item \ref{item:disjunctions:theorem:boxGLPSound}), we find $\chi$ with $\bDiamond{\beta}\chi \ \leftrightarrow \ \bDiamond{\beta} \varphi' \wedge \bDiamond{\beta}\psi'$ whence by the reflexive induction hypothesis also $\Box \big ( \bDiamond{\beta}\chi \ \leftrightarrow \ \bDiamond{\beta} \varphi' \wedge \bDiamond{\beta}\psi'\big )$. Consequently, we have that $\Box(\la \beta \ra \chi \to \varphi \wedge \psi)$ and we are done with the direction $[ \alpha] \varphi \wedge [ \alpha ] \psi \to [\alpha](\varphi \wedge \psi)$. The other direction follows directly from Lemma \ref{theorem:boxBoxExfalso} since $\Box \big( (\varphi \wedge \psi) \to \varphi \big)$ and $\Box \big( (\varphi \wedge \psi) \to \psi \big)$.

{\bf Item \ref{item:distributivity:theorem:boxGLPSound}}: $\forall \varphi \, \forall \psi \ \Big( \bBox{\alpha}_T^\Lambda(\varphi \to \psi) \to (\bBox{\alpha}_T^{\Lambda}\varphi \to \bBox{\alpha}_T^{\Lambda}\psi)\Big)$.
From the previous item we know that 
\[
[{\alpha}] (\varphi\to \psi) \wedge [{\alpha}]\varphi \ \leftrightarrow \ [{\alpha}]\Big ( (\varphi\to \psi) \wedge \varphi \Big)
\]
so that the result follows from Lemma \ref{theorem:boxBoxExfalso}.

{\bf Item \ref{item:disjunctions:theorem:boxGLPSound}}: $\forall \varphi\, \forall \psi \, \exists \chi \ \Big(\bBox{\alpha}_T^{\Lambda}\varphi \vee \bBox{\alpha}_T^{\Lambda}\psi \ \ \leftrightarrow \bBox{\alpha}_T^{\Lambda}\chi \Big)$.
We still reason in $T$ and assume that for some arbitrary $\varphi$ and $\psi$ we have $ [{\alpha}] \varphi$ or $ [{\alpha}] \psi$.  By Corollary \ref{theorem:FirstTwoPredicatesAreGood} we may assume that $\alpha \geq 2$ (observe that our assumption that $T$ proves transfinite $\Pi_2^0([\alpha]^\Lambda_T)$ induction, implies that certainly $T$ proves $\Sigma^0_1$ collection). Under this assumption we make a case distinction.

In case that $\nboxBox{\alpha} \bot$ we see by Lemma \ref{theorem:boxBoxExfalso} that for any formula $\chi$ we have $\nboxBox{\alpha}\chi \ \leftrightarrow \ \big( \nboxBox {\alpha} \varphi \vee \nboxBox {\alpha} \psi \big)$ so that equivalence certainly holds for the $\chi$ we propose in the alternative case.

That is, we consider the case that $\nboxDiamond{\alpha} \top$.  We claim that under this assumption, e.g.~$\nboxBox \alpha \varphi$ is equivalent to the single $\exists \, \beta {\prec} \alpha \, \exists \varphi' \ \big(\nboxDiamond {\beta} \varphi' \wedge \Box (\nboxDiamond{\beta} \varphi' \to \varphi) \big)$. But this is clear since by definition $\nboxBox \alpha \varphi$ is equivalent to 
\[
\Box \varphi \ \vee \ \exists \, \beta {\prec} \alpha \, \exists \varphi' \ \big(\nboxDiamond {\beta} \varphi' \wedge \Box (\nboxDiamond{\beta} \varphi' \to \varphi) \big)
\]
so we only need to see that the first disjunct $\Box \varphi$ implies the second. But since we work under the assumption that $\nboxDiamond \alpha \top$, in particular, we have $\nboxDiamond \beta \top$ for any ordinal $\beta \prec \alpha$. Moreover, for any such $\beta$ we have that $\Box \varphi \to \Box (\nboxDiamond \beta \top \to \varphi)$ so that the claim follows.

Using this observation, we find by unfolding the definition of 1-M\"unchhausen provability in $\nboxBox{\alpha} \varphi \vee \nboxBox{\alpha} \psi$ some formulas $\varphi'$ and $\psi'$ and some ordinals $\beta,\beta'<\alpha$  so that 
\begin{equation}\label{equation:disjunctionUnfoldedFirstStep}
\nboxDiamond {\beta} \varphi' \wedge \Box (\nboxDiamond{\beta} \varphi' \to \varphi)\ \mbox{ or } \ \nboxDiamond{\beta'} \psi' \wedge \Box (\nboxDiamond{\beta'} \psi' \to \psi).
\end{equation}

Since we work under the assumption that $\nboxDiamond{\alpha} \top$ holds with $\alpha \geq 2$, we certainly have $\nboxDiamond{\max{\{ \beta,\beta', 1 \}}}\top$ so that as before we may and will assume without loss of generality that $\beta'=\beta$ and $\beta\geq 1$.
Using the distributivity laws we see that \eqref{equation:disjunctionUnfoldedFirstStep} is equivalent to
\begin{equation}\label{equation:diamondFormulasInClosureUnderDisjunctionProof}
\Big ( \nboxDiamond{\beta} \varphi' \vee \nboxDiamond{\beta} \psi' \Big ) \ \wedge \  \Big( \nboxDiamond{\beta} \varphi' \vee\, \Box (\nboxDiamond{\beta} \psi' \to \psi) \Big) \ \wedge \  \Big( \nboxDiamond{\beta} \psi' \vee \, \Box (\nboxDiamond{\beta} \varphi' \to \varphi) \Big)
\end{equation}
and,
\begin{equation}\label{equation:boxFormulasInClosureUnderDisjunctionProof}
\Box (\nboxDiamond{\beta} \varphi' \to \varphi) \ \ \vee \ \ \Box (\nboxDiamond{\beta} \psi' \to \psi).
\end{equation}
By the reflexive induction hypotheses and by Lemma  \ref{theorem:basicGLPlemmas}.\ref{item:boxDisjunctionDiamond:theorem:basicGLPlemmas} --by the inductive hypothesis and Lemma \ref{theorem:GLPconservativelyExtendsFragments} we may use any $\glp_\beta$ reasoning-- 
we see that \eqref{equation:diamondFormulasInClosureUnderDisjunctionProof} can be written as a single diamond formula, say $\nboxDiamond{\beta} \chi'$. Thus, we would be done if we can find some formula $\chi$ so that 
\begin{equation}\label{equation:existsChiRequirementInClosureDisjunctionProof}
\Big ( \Box (\nboxDiamond{\beta} \varphi' \to \varphi) \ \ \vee \ \ \Box (\nboxDiamond{\beta} \psi' \to \psi)\Big ) \ \ \leftrightarrow \ \  \Box(\nboxDiamond{\beta}\chi' \to \chi).
\end{equation}
We will find such a $\chi$ by applying the FGH theorem with base theory $T+\la \beta \ra \chi'$. It thus remains to see that this theory $T+\la \beta \ra \chi'$ is consistent.

From $\nboxDiamond{\beta}\chi'$ we get by negative introspection (Lemma \ref{theorem:crossAxiomsSound}.\ref{item:negativeIntrospection:theorem:crossAxiomsSound}) that $\nboxBox{\alpha}\nboxDiamond{\beta}\chi'$. Recall that we work under the assumption that $\nboxDiamond{\alpha}\top$ so that by distributivity at level $\alpha$ --which is already known at this stage in our proof-- we get 
\[
\nboxDiamond{\alpha}\top \wedge \nboxBox{\alpha}\nboxDiamond{\beta}\chi' \to \nboxDiamond{\alpha}\nboxDiamond{\beta}\chi'
\]
whence by monotonicity we get $\Diamond \boxDiamond{\beta}\chi'$ whence $\Diamond_{T+ \nboxDiamond{\beta}\chi'} \top$. 

The existence of some $\chi$ so that 
\eqref{equation:existsChiRequirementInClosureDisjunctionProof} holds is now guaranteed by the (formalized) FGH theorem applied to the theory $T + \boxDiamond{\beta}\chi'$ since 
\[
\Box (\boxDiamond{n} \varphi' \to \varphi) \ \ \vee \ \ \Box (\boxDiamond{n} \psi' \to \psi) \in \Sigma^0_1.
\]

{\bf Item \ref{item:transitivity:theorem:boxGLPSound}}: $\forall \varphi\ \Big ( \bBox{\alpha}_T^{\Lambda} \varphi \to \bBox{\alpha}_T^{\Lambda} \bBox{\alpha}_T^{\Lambda}\varphi \Big )$.
While reasoning in $T$ we assume $\nboxBox{\alpha}\varphi$ and only consider the non-trivial case. Thus, for some $\varphi'$ and some $\beta\prec \alpha$ we get $\nboxDiamond {\beta}\varphi'$ and $\Box(\nboxDiamond {\beta}\varphi' \to \varphi)$. By negative introspection we get $\nboxBox{\alpha} \nboxDiamond{\beta}\varphi'$. Since $T$ is a 1-M\"unchhausen theory it proves some properties of the order $\prec$. In particular, from $\beta\prec\alpha$, we also get $\nboxBox{\alpha} (\beta\prec \alpha)$. From $\Box(\nboxDiamond {\beta}\varphi' \to \varphi)$ we obtain by applying successively provable $\Sigma^0_1$ completeness and monotonicity that $\nboxBox{\alpha}\Box(\nboxDiamond {\beta}\varphi' \to \varphi)$. Since we already proved closure of the $[\alpha]$ predicate under conjunctions, we can collect all the information under the $\nboxBox{\alpha}$ and applying Lemma \ref{theorem:crossAxiomsSound}.\ref{item:existsGoesInTheBox:theorem:boxBoxExfalso} we see that we have obtained $\nboxBox{\alpha}\nboxBox{\alpha}\varphi$.
\end{proof}

\section{Completeness of M\"unchhausen provability}\label{section:completeness}

In this section we shall prove that under some modest set of extra assumptions, we can obtain completeness of one-M\"unchhausen provability. Basically, this section consist of invoking a result from \cite{FernandezJoosten:2018:OmegaRuleInterpretationGLP} and recasting it in our context. Let us first recall some definitions and results.

\subsection{Uniform proof and provability predicates}
The definitions and results from this subsection all come from \cite{FernandezJoosten:2018:OmegaRuleInterpretationGLP} where an arithmetical completeness proof is given that is schematic in an abstract kind of provability predicates. A first step in defining these provability predicates consists of defining so-called \emph{$\Lambda$-uniform proof and provability predicates over $T$}.

\begin{definition}\label{UProv}
Let $T$ be representable and $\Lambda$ a linear order. Given a formula $\pi(c,\lambda,\phi)$, we introduce the notation $\provXc c\lambda \pi{\phi}=\pi(c,\lambda,\phi)$, as well as $\provx \lambda\pi\phi=\exists c \provXc c \lambda \pi{\phi}$. The dual notions $\consXc c\lambda\pi{\phi}$ and $\consx {\xi}\pi{\phi}$ are defined as $\neg \pi(c,\lambda,\neg \phi)$ and $\neg \exists c \provXc c \lambda \pi{\neg \phi}$ respectively.

A {\em $\Lambda$-uniform proof predicate over $T$} is a formula $\pi(c,\lambda,\phi)$ (with all free variables shown) satisfying
\begin{enumerate}
\item 
$T\vdash {\mathrm{I}\Sigma_{1}^0}(\pi)$;\label{UProv1}

\item 
$T\vdash\forall \lambda\forall\phi\ (\nc_T\phi\rightarrow \provx\lambda\pi\phi)$;\label{UProv2}

\item 
$T\vdash\forall \lambda\forall \phi\forall\psi \ \Big(\provx \lambda\pi(\psi\to\phi)\wedge\provx \lambda\pi\psi\rightarrow \provx \lambda\pi\phi\Big)$;\label{UProv3}

\item 
$T\vdash\forall c\, \forall \lambda \, \forall \xi{\leqLam}\lambda\, \forall \phi\ \Big(\provXc c\xi \pi{\phi}\rightarrow \provXc c\lambda \pi \phi\Big)$;\label{UProvNew}

\item 
$T\vdash\forall c\, \forall \lambda\, \forall \phi\ \Big(\provXc c\lambda\pi{\phi}\rightarrow \provx\lambda \pi \provXc{\dot c}{\dot\lambda}{\pi}{\dot\phi}\Big)$;\label{UProv4}

\item 
$T\vdash\forall c \forall \lambda\, \forall \phi\ \Big(\consXc c\lambda\pi{\phi}\rightarrow \provx\lambda \pi \consXc{\dot c}{\dot\lambda}{\pi}{\dot\phi}\Big)$;\label{UProv5}

\item 
$T\vdash\forall \lambda\, \forall\,  \xi{\leLam}\lambda\, \forall \phi\ \Big ( \consx {\xi}\pi{\phi}\rightarrow \provx{\lambda}\pi{{\consx {\dot\xi}{\pi}{\dot\phi}}}\Big)$.\label{UProv6}
\end{enumerate}

We say that $\pi$ is {\em sound}\footnote{Observe that for $\pi$ to be sound, we must have that $T$ itself was already sound.} if, moreover, $\mathbb N\models \forall \lambda\forall \phi\ (\provx\lambda\pi\phi\rightarrow \phi)$.

A formula $\hat\pi$ is a {\em $\Lambda$-uniform provability predicate} over $T$ if $T\vdash \hat\pi\leftrightarrow \exists c\ \pi$, where $\pi$ is a $\Lambda$-uniform proof predicate.
\end{definition}

Moreover, the provability predicates are required to require a modicum of good behaviour as captured in the following definition.

\begin{definition}
Let $\pi$ be a $\Lambda$-uniform proof predicate over a theory $T$. We say that $\pi$ is {\em normalized} if it is provable in $T$ that for every $\lambda$ we have that every $\lambda$-derivable formula has infinitely many $\lambda$-derivations and, whenever $\provXc c\lambda \pi\phi$ and $\provXc c\lambda \pi\psi$, it follows that $\phi=\psi$; in other words, every derivation must be a derivation of a single formula.
\end{definition}

Modal formulas are linked to arithmetical ones via an arithmetic interpretation.

\begin{definition}
An {\em arithmetic interpretation} is a function\footnote{By $\mathbb P$ we denote the set of propositional variables and by $\sentences$ we denote the set of $\Pi^1_\omega$ sentences.} $f:\mathbb P\to \sentences$.

If $\pi$ is a $\Lambda$-uniform proof predicate over $T$, we denote by $f_\pi$ the unique extension of $f$ such that $f_\pi(p)=f(p)$ for every propositional variable $p$, $f_\pi(\bot)=\bot$, $f_\pi$ commutes with Booleans and $f_\pi([\lambda]\phi)=[\overline\lambda]_\pi f_\pi(\phi)$.
\end{definition}

The following uniform completeness theorem is proven in \cite[Theorem 10.2]{FernandezJoosten:2018:OmegaRuleInterpretationGLP} and provides us with an easy way to prove completeness for our current interpretation.

\begin{theorem}\label{complete}
If $\Lambda$ is a computable linear order, $T$ is any sound, representable theory extending $\rca$, $\pi$ is a sound, normalized, $\Lambda$-uniform proof predicate over $T$ and $\phi$ is any $\mlang$-formula, ${\sf GLP}_\Lambda\vdash \phi$ if and only if, for every arithmetic interpretation $f$, $T\vdash f_\pi(\phi)$.
\end{theorem}

\subsection{Arithmetical completeness for M\"unchhausen provability}

We can now combine the results from this paper and the previous subsection to see that under some extra conditions we obtain arithmetical completeness for one-M\"unchhausen provability.

\begin{theorem}[Arithmetical Completeness]\label{theorem:MunchhausenArithmeticalComplete}
Let $\Lambda$ be a computable linear order, $T$ is any sound, representable one-M\"unchhausen theory extending $\rca$ with corresponding provability predicate ${\nboxBox{\alpha}_T}^\Lambda \varphi$ so that $T\vdash {\mathrm{I}\Sigma_{1}^0}({\nboxBox{\alpha}_T}^\Lambda \varphi)$. We then have that ${\nboxBox{\alpha}_T}^\Lambda \varphi$ is a uniform provability predicate and in particular,
\[
\glp_\Lambda \vdash \varphi \ \ \Longleftrightarrow \ \ \forall * \ T \vdash \varphi^*.
\]
\end{theorem}

\begin{proof}
As always, the $*$ in the statement of the theorem is understood to range over arithmetical interpretations that map propositional variables to arbitrary sentences, so that $*$ commutes with the boolean connectives and each modal formula $[\alpha] \psi$ is mapped to ${{\nboxBox{\overline \alpha}_T}}^\Lambda \psi^*$.

From our provability predicate (omitting sub and superscripts) $\nboxBox \alpha \varphi$ we will define a proof predicate $\pi (c,\lambda,\phi)$ for which we will observe that over $T$ it is a normalized uniform proof predicate so that provably $\exists c\ \pi (c,\lambda,\phi) \ \leftrightarrow \ \nboxBox \lambda \phi$. To this end we define
\[
\pi (c,\lambda,\phi) \ := \ c=\la c_0, c_1\ra \wedge \left \{ 
\begin{array}{lcll}
\Big ( c_0 =0 &\wedge& {\sf Proof}_T(c_1, \phi) \Big) & \vee\\
 \Bigg( c_0 =1 &\wedge & c_1 = \la \xi, \psi, p\ra \ \wedge \ \xi \prec \lambda \ \ \ \wedge & \\
 & & \boxDiamond \xi \psi \wedge {\sf Proof}_T(p, \boxDiamond \xi \psi \to \phi)\Bigg).\\
\end{array}
\right . 
\]
It is straightforward to see that, indeed, $T \vdash \exists c\ \pi (c,\lambda,\phi) \ \leftrightarrow \ \nboxBox \lambda \phi$. Since ${\sf Proof}_T$ is a normalized proof predicate, so is $\pi$. Thus, we should only check Properties 1 -- 7 from Definition \ref{UProv}. Property 1 is one of the assumptions of the theorem and Properties 2, 3 and 7 follow directly from the arithmetical soundness of one-M\"unchhausen provability. Property 4 follows since $T$ is a one-M\"unchhausen theory whence proves transitivity of $\prec$. Properties 5 and 6 are a direct consequence of the definition of $\pi$ and the soundness of the one-M\"unchhausen provability predicate.

\end{proof}

\section{Some notes on the Formalisation of one-M\"unchhausen provability}\label{section:formalisation}

Throughout this paper we have been talking about M\"unchhausen provability predicates and proving all sorts of properties of them. The reserved reader may now question whether there exist one-M\"unchhausen theories with corresponding one-M\"unchhausen provability predicates at all. In this section we sketch how to formalize a M\"unchhausen provability predicate in second order arithmetic.

Just as in \cite{FernandezJoosten:2018:OmegaRuleInterpretationGLP} we start our formalization by reserving a set parameter $X$ where we will collect all the pairs $\la \alpha, \varphi \ra$ of ordinals $\alpha$ and formulas $\varphi$ so that $[\alpha] \varphi$ holds. Next, we will write down a predicate that all and only the correct pairs $\la \alpha, \varphi \ra$ are in $X$. Thus, we write the recursion for one-M\"unchhausen provability replacing every occurrence of $[ \alpha ] \varphi$ by $\la \alpha , \varphi \ra \in X$ and consequently replacing $\la \alpha \ra \varphi$ by $\la \alpha, \neg \varphi\ra \notin X$. We define any set satisfying our predicate to be an  $1{-}\sf IMC$ for \emph{Iterated one-M\"unchhausen Class}.

By naively doing so, a problem arises namely that we get occurrences of the set variable $X$ under the regular provability predicate $\Box_T$. By using numerals we can speak under a box about numbers that `live outside the box'. However, we do not have any syntactical artefact to denote arbitrary sets. A possibly way out here would be to resort to \emph{oracle-provability} as introduced in \cite{CordonFernandezJoostenLara:2017:PredicativityThroughTransfiniteReflection}. Thus, for one-M\"unchhausen provability, the predicate would look something like:
\[
\begin{array}{lll}
1{-}{\sf IMC}(X,\alpha):= & & \\
\ \ \ \forall \, \xi{\leq}\alpha\, \forall \varphi \Big[
\la \xi, \varphi \ra \in X   &\leftrightarrow & \Big( \Box_T \varphi \, \vee \, \exists \psi \, \exists \, \zeta{<}\xi \ \big( \, \la \zeta, \neg \psi \ra \notin X \ \wedge \\
 & & \ \ \ \ \ \ \  \Box_{T | X } (\la \zeta, \neg \psi \ra \notin X \to \varphi) \big) \Big)\Big].
\end{array}
\]
With such a predicate we can then define:
\[
[\alpha]_{T,1} \varphi := \forall X \Big( 1{-}{\sf IMC}(X,\alpha) \to \la \alpha, \varphi \ra \in X\Big).
\]
However, it is not clear if such a predicate will satisfy the required recursive equation since the relation between oracle provability and regular provability is not yet entirely understood in all its details.

For these and other reasons we choose a different approach. We will anticipate that hopefully/probably the $1{-}\sf IMC$ predicate will define a unique set. Then, under the box we can just use any set that satisfies ${\sf IMC}(X)$. Of course, the fixpoint theorem allows us to do so.  In the formalisation of M\"unchhausen provability we will closely follow \cite{FernandezJoosten:2018:OmegaRuleInterpretationGLP}. As such we allow ourselves to be rather sketchy and refer to \cite{FernandezJoosten:2018:OmegaRuleInterpretationGLP} for the details.

\begin{definition}
We define the predicate $1{-}{\sf IMC}(X,\gamma)$ using the fixpoint theorem so that it satisfies (provably in \eca) the following recursion.
\[
\begin{array}{lll}
1{-}{\sf IMC}&\hspace{-.2cm}(X,\gamma) \ \longleftrightarrow& \\
 &  \Big( \forall\, \alpha {\preceq}\gamma\ \forall \varphi\ \Big[ & \la \alpha , \varphi\ra \in X \ \leftrightarrow \\
 & & \Box_U \varphi \vee \exists\, \beta{\prec}\alpha \exists \psi \Big( \la \beta, \neg \psi \ra\notin X \wedge \\
 & & \ \ \ \Box_U \big[  \exists X (
 1{-}{\sf IMC} (X,\dot \beta) \wedge \la\dot \beta, \neg \dot \psi \ra \notin X) \to \varphi\big]\Big)\Big]\Big)
\end{array}
\]
With this Iterated one-M\"unchhausen Class predicate we define our one-M\"unchhausen predicate as
\[
[\alpha]_U \varphi := \forall X \Big( 1{-}{\sf IMC}(X,\alpha) \to \la \alpha, \varphi \ra \in X\Big).
\]
\end{definition}
It is clear that our definition supposes that we fix an ordinal notation system for some ordinal $\Lambda$ and that all our ordinal quantifications are restricted to this $\Lambda$. 
We observe that 
\[
\la\alpha\ra \varphi := \exists X \Big( 1{-}{\sf IMC}(X,\alpha) \wedge \la \alpha, \neg \varphi \ra \notin X\Big).
\]
Consequently we can rewrite the defining recursion for Iterated one-M\"unchhausen Classes as
\[
\begin{array}{lll}
1{-}{\sf IMC}&\hspace{-.2cm}(X,\gamma) \ \longleftrightarrow& \\
 &  \Big( \forall\, \alpha {\preceq}\gamma\ \forall \varphi\ \Big[ & \la \alpha , \varphi\ra \in X \ \leftrightarrow \\
 & & \Box_U \varphi \vee \exists\, \beta{\prec}\alpha \exists \psi \big( \la \beta, \neg \psi \ra\notin X \wedge \\
 & & \ \ \ \Box_U \big[ \la \beta \ra \psi \to \varphi\big]\big)\Big]\Big).
\end{array}
\]
It is clear that $1{-}{\sf IMC}$ depends on the base theory $U$ and on the ordinal representation $\Lambda$ but for the sake of readability we suppress these dependencies in our notation. We remark that $1{-}{\sf IMC} (X,\gamma)$ is of complexity $\Pi^0_2$ with free set variable $X$. Our predicate $[\alpha]\varphi$ has a universal quantifier ranging over all sets that are iterated M\"unchhausen classes. Of course, we would hope that indeed such classes are uniquely defined if they exists at all.

In order to express this, we will fix the following notation
\[
X\equiv_\alpha Y \ \ := \ \ \forall \, \beta{\preceq}\alpha\, \forall  \varphi \Big( \la \beta, \varphi \ra \in X \ \longleftrightarrow \ \la \beta, \varphi \ra \in Y  \Big),
\]
and
\[
\exists^{\leq 1} X \ {\sf IMC}(X,\alpha) \ \ := \ \ \forall X\, \forall Y\ \Big( {\sf IMC} (X,\alpha) \wedge {\sf IMC} (Y,\alpha) \ \longrightarrow \ X\equiv_\alpha Y\Big).
\]
We can now state and prove a key ingredient in proving that our formalisation satisfies the defining recursion for M\"unchhausen provability.

\begin{lemma}\label{theorem:uniquenessIMCs}
Let $U$ be a theory extending \eca. We have that
\[
\AcaNaught + {\sf wo}(\alpha) \vdash \forall \alpha \, \exists^{\leq 1} X \ {\sf IMC}(X,\alpha).
\]
\end{lemma}

\begin{proof}
We prove by transfinite induction that ${\sf IMC} (X,\alpha) \wedge {\sf IMC} (Y,\alpha) \to  X\equiv_\alpha Y$ where $X$ and $Y$ are unbounded set variables. Note that this is an arithmetical formula so that 
 \AcaNaught can prove transfinite induction up to $\alpha$ for this formula since we assumed ${\sf wo}(\alpha)$.
\end{proof}

Now that we have uniqueness we proceed as in \cite[ Theorem 4.3]{FernandezJoosten:2018:OmegaRuleInterpretationGLP} to observe that we actually may perform transfinite induction for second order formulas as long as the second order formulas are restricted to the ${\sf IMC}$s.

\begin{theorem}\label{TheoUnique} Given a formula $\theta(X)\in {\bm \Pi}^1_\omega$,
\[
\base\vdash \forall \Lambda \ \Big( \exists^{\leq 1} X\ \theta(X)\wedge{\tt wo}(\Lambda)\rightarrow {\tt TI}(\Lambda,\GuardLang\theta)\Big).\]
\end{theorem}

We are now ready to prove that our formalisation satisfies the required recursion.

\begin{theorem}\label{theorem:RecursionFormalized}
Let $T$ be any presentable theory extending \eca. We have
\[
\begin{array}{lr}
\AcaNaught + {\sf wo}(\beta) + \exists X 1{-}{\sf IMC}(X, \beta) \vdash & \forall \, \alpha{\preceq}\beta \ \Big[
\nboxBox{\alpha}_T \varphi \ \leftrightarrow \ \Box_T \varphi \vee \exists \psi \, \exists \, \gamma \Big( \gamma \prec \alpha \wedge \ \\
  &  \
  \nboxDiamond{\gamma}_T \psi \, 
\wedge \ \ \ \\
&  \Box_T\big(  \nboxDiamond{\gamma}_T \psi \, \to \, \varphi \big) \Big) \Big]. 
\end{array}
\]
\end{theorem}

\begin{proof}
By transfinite induction on $\alpha$ as in \cite{FernandezJoosten:2018:OmegaRuleInterpretationGLP}. Note that we need the existence of a $1{-}{\sf IMC}$ for the $\to$ direction. By Theorem \ref{TheoUnique} we have access to the transfinite induction in \AcaNaught since we proved uniqueness for $1{-}\sf IMC$'s.
\end{proof}

\section{Weakening the base theory: M\"unchhausen provability}\label{section:MunchhausenProvabilityMultipleOracle}

In this paper we have introduced the notion of one-M\"unchhausen provability for which we have proven arithmetical sound and completeness. Furthermore, we have shown in Theorem \ref{theorem:RecursionFormalized} that the notion can be formalised in second order arithmetic. However, the theory where the formalisation takes place is quite strong. In particular, it requires a fair amount of transfinite induction. As pointed out, this prove theoretic strength is consequently also required in the object theory which is not desirable. Via various tricks, one can lower the required proof theoretic strength of the object and meta-theory. A first step in doing so is via the introduction of \emph{M\"unchhausen provability}. Further tricks are presented and worked out in \cite{Joosten:2019:TransfiniteTuringjumps}.

To define M\"unchhausen provability we will start out with a very similar but slightly different recursion equivalence:
\begin{multline}\label{equation:FullVectorBox}
\vboxBox{\alpha}_T \varphi \ \leftrightarrow \ \Box_T \varphi \vee \exists \sigma \, \exists \, \tau \ \Big( \ |\sigma | = |\tau | \ \wedge \ \forall 
\, i{<}| \tau |\, \tau_i {\prec}\alpha \ \wedge \  \forall \, i{<}| \sigma |\, \vboxDiamond{\tau_i}_T \sigma(i) \, \ \ \ \ \ \ \\
\wedge \, \Box_T\big( \forall \, i{<}| \sigma |\, \vboxDiamond{\tau_i}_T \sigma(i) \, \to \, \varphi \big) \Big). 
\end{multline}
In this recursive equivalence we understand that $\sigma$ is a finite sequence of formulas with $|\sigma |$ denoting the length of the sequence and $\sigma(i)$ denoting the $i$th element of the sequence. Likewise, $\tau$ is understood as being a sequence of ordinals all bounded by $\alpha$. We will write either $\tau(i)$ or $\tau_i$ for the $i$th element of $\tau$. Moreover, $\vboxDiamond \alpha$ is as always to be read a shorthand for $\neg \vboxBox  \alpha \neg$. 

One of the main complications in proving the arithmetical soundness of one-M\"unchhausen provability in the previous section was in the proof of the closure of provability under conjunctions that is, $\nboxBox{\alpha} \varphi \wedge  \nboxBox \alpha \psi \leftrightarrow \nboxBox \alpha (\varphi \wedge \psi)$. The proof of this required a weak closure of consistency under conjunctions --$\forall \varphi, \psi \, \exists \chi \Big( \nboxDiamond \alpha \varphi \wedge  \nboxDiamond \alpha \psi \leftrightarrow \nboxDiamond \alpha \chi\Big)$-- so that the conjunction of two oracle sentences could be conceived as a single oracle sentence. However, in the new recursive equivalence as we just defined in \eqref{equation:FullVectorBox}, the closure of oracles under conjunctions is built into the definition.

A further complication in proving the arithmetical soundness of one-M\"unch-hausen provability in the previous sections was caused by the fact that weak closure under conjunctions of consistency needed to be verified under a box. This was obtained by requiring a fair amount of transfinite induction and by requiring that the object and meta-theory be equal. In this last section we shall see that these requirements can also be circumvented. 

The defining equation \eqref{equation:FullVectorBox} begs for a notational simplification. From now on, the greek letter $\sigma$ shall be reserved to denote sequences of formulas and the greek letter $\tau$ shall be reserved to denote sequences of ordinals. As such, we settle upon the notational convention that $\tau\prec\alpha$ is short for $\forall \, i{<}| \tau |\, \tau_i {\prec} \alpha$ and $\vboxDiamond{\tau}_T \sigma$ is short for $|\sigma| = |\tau| \ \wedge \  \forall \, i{<}| \sigma |\, \vboxDiamond{\tau_i}_T \sigma(i)$. Since we shall require that provably $|\sigma| = |\tau| \to \Box_T |\sigma| = |\tau|$, the defining recursion can be recasted as 
\begin{equation}\label{equation:vectorBox}
\vboxBox{\alpha}_T \varphi \ \leftrightarrow \ \Box_T \varphi \vee \exists \sigma \, \exists \, \tau{\prec}\alpha\ \Big(\vboxDiamond{\tau}_T \sigma \, \wedge \, \Box_T\big( \vboxDiamond{\tau}_T \sigma \, \to \, \varphi \big) \Big). 
\end{equation}

Although we still cannot prove that different predicates that provably satisfy \eqref{equation:vectorBox} are provably equivalent, at least proving soundness of $\glp_\Lambda$ for such predicates becomes an easy matter.  Let us first define some important notions as before but now for M\"unchhausen provability instead of one-M\"unchhausen provability.

\begin{definition}
Let us call a theory $T$ a $\Lambda$-M\"unchhausen theory whenever we can define a predicate ${\vboxBox {\alpha}_T}^\Lambda$ so that $T$ proves \eqref{equation:vectorBox} together with 
\[
\begin{array}{l}
T\vdash ``{\prec}\mbox{ is transitive, right-discrete and has a minimal element}",\\
T\vdash (\xi \prec \zeta) \to {\vboxBox{\zeta}_T}^\Lambda (\xi \prec \zeta),\\
\mbox{$\xi < \zeta < \Lambda$ implies $T\vdash \xi \prec \zeta$.}
\end{array}
\]
Moreover, it is understood that $T$ has a simple coding machinery for finite sequence of objects so that the obvious facts about length and concatenation provably hold. For example, $T \vdash |\tau| =n \to \Box_T |\tau| =n$, etc.

In this case we call ${\vboxBox {\alpha}_T}^\Lambda$ a $T(\Lambda)$-M\"unchhausen provability predicate.
\end{definition}

When the theory $T$ and ordinal $\Lambda$ are clear from the context, we shall simply speak of a M\"unchhausen theory and of a M\"unchhausen provability predicate. On occasion we might only mention the ordinal $\Lambda$ or only the theory $T$ and speak of, for example, a $\Lambda$-M\"unchhausen theory and a $T$-M\"unchhausen provability predicate respectively. As with one-M\"unchhausen provability we see that the interaction axioms become trivial to prove for any M\"unchhausen provability predicate. In what follows we will revisit and simplify the soundness proof.

\begin{lemma}\label{theorem:MunchhausenCrossAxioms}
Let $T$ be a $\Lambda$-M\"unchhausen theory with corresponding predicate ${\vboxBox{\alpha}_T}^\Lambda$. Omitting sub and superscripts, we have that
\begin{enumerate}
\item
$T\vdash \forall \alpha\, \forall \varphi \, \forall \, \beta {\prec} \alpha {\prec}\Lambda\ \big(  \vboxBox \beta \varphi \to  \vboxBox \alpha \varphi \big)$,

\item
$T\vdash \forall \alpha\, \forall \varphi \, \forall \, \beta {\prec} \alpha {\prec}\Lambda\ \big( \vboxDiamond \beta \varphi \to \vboxBox  \alpha \vboxDiamond \beta \varphi \big)$, and more in general

\item
$T\vdash \forall \alpha\, \forall \sigma \, \forall \, \tau {\prec} \alpha {\prec}\Lambda\ \big(  \vboxDiamond \tau \sigma \to \vboxBox  \alpha  \vboxDiamond \tau \sigma \big)$.
\end{enumerate}
\end{lemma}

\begin{proof}
The proof is straightforward and completely analogous to the proof of Lemma \ref{theorem:crossAxiomsSound}. Let us just shortly comment on the second item. So, we reason in $T$ and pick a formula $\varphi$ and ordinals $\alpha$ and $\beta$ as indicated, assuming $\vboxDiamond \beta \varphi$. We now consider the sequence $\sigma_\varphi$ of length 1 whose only element is the formula $\varphi$. Likewise, we consider the sequence $\tau_\beta$ of length 1 whose only element is the ordinal $\beta$. Clearly, $T \vdash \vboxDiamond {\tau_\beta} \sigma_\varphi \to \vboxDiamond \beta \varphi$ so that $\vboxBox \alpha \vboxDiamond \beta \varphi$ follows.
\end{proof}

Contrary to the case of  1-M\"unchhausen provability it becomes now an easy exercise to see that each (internally quantified) provability predicate satisfies the distribution axioms for the basic modal logic \logic K. Moreover, necessitation is also a routine matter. Before we prove this, we first need a technical easy lemma similar to Lemma \ref{theorem:boxBoxExfalso} whose proof is immediate.

\begin{lemma}\label{theorem:MunchhausenClosedUnderZeroMP}
Let $T$ be a $\Lambda$-M\"unchhausen theory with corresponding predicate ${\vboxBox{\alpha}_T}^\Lambda$. Again, omitting sub and superscripts, we have that
\[
U\vdash \forall \, \alpha{\prec}\Lambda\, \forall\, \varphi, \psi, \chi\ \Big( \,  \vboxBox \alpha \psi \wedge \Box \varphi \wedge \Box (\varphi \wedge \psi \to \xi) \ \to \ \vboxBox \alpha \xi \,\Big). 
\]
\end{lemma}

With this technical lemma at hand it becomes very easy to see that each M\"unchhausen provability predicate ${\vboxBox{\alpha}_T}^\Lambda$ defines a normal\footnote{It is in this lemma that we see that working with a single $\beta$ would not have worked directly. That is, if we had defined $\vboxBox{\alpha}_T \varphi \ \leftrightarrow \ \Box_T \varphi \vee \exists \sigma \, \exists \, \beta{\prec}\alpha \ \Big( \forall \, i{<}| \sigma |\, \vboxDiamond{\beta}_T \sigma(i) \, \wedge \, \Box_T\big( \forall \, i{<}| \sigma |\, \vboxDiamond{\beta}_T \sigma(i) \, \to \, \varphi \big) \Big)$. The distributivity axiom can then only be proved if we can work with the largest consistency statement. Thus, something like Lemma \ref{theorem:basicGLPlemmas}.\ref{item:bigConsEquivSmallCons:theorem:basicGLPlemmas} should be available. For that, the soundness of $\glp_\beta$ would be needed and we are back at the transfinite induction template again.} modal logic. 

\begin{lemma}\label{theorem:distributionMunchhausen}
Let $T$ be a $\Lambda$-M\"unchhausen theory with corresponding predicate ${\vboxBox{\alpha}_T}^\Lambda$. Again, omitting sub and superscripts, we have that

\begin{enumerate}
\item
$T \vdash \forall \, \alpha{\prec}\Lambda \, \forall \varphi, \forall \psi \ \Big( \vboxBox \alpha (\varphi \to \psi) \ \to\ \big( \vboxBox \alpha \varphi \to \vboxBox \alpha \psi \big) \Big)$, and

\item
for any ordinal $\alpha$ below $\Lambda$, if $T\vdash \varphi$, then $T\vdash \vboxBox \alpha \varphi$.
\end{enumerate}
\end{lemma}

\begin{proof}
The proof of the second item is easy and identical to the proof Lemma \ref{theorem:necessitationSoundForOneMunchhausenProvability}. It is in the first item where we see that working with sequences of formulas instead of formulas in our oracles is essential. So, let us reason in $T$ and fix $\alpha$ and $\varphi$ as stated. We assume $\vboxBox \alpha (\varphi \to \psi)$ and $ \vboxBox \alpha \varphi$ and need to prove $\vboxBox \alpha \psi$. 

The case that both $\Box (\varphi \to \psi)$ and $\Box \varphi$ hold is trivial and in case one of them holds, Lemma \ref{theorem:MunchhausenClosedUnderZeroMP} provides a proof. 

So, in the remaining and only non-trivial case, we find two pairs of sequences $\sigma_\varphi$ with $\tau_\varphi$ and $\sigma_{\varphi\to\psi}$ with $\tau_{\varphi\to\psi}$ so that 
$\tau_\varphi\prec \alpha \wedge \la \tau_\varphi \ra \sigma_\varphi  \ \wedge \ \Box (\la \tau_\varphi \ra \sigma_\varphi\to \varphi)$ 
and also
$\tau_{\varphi\to\psi}\prec \alpha \wedge \la \tau_{\varphi\to\psi} \ra \sigma_{\varphi\to\psi}  \ \wedge \ \Box \Big(\la \tau_{\varphi\to\psi} \ra \sigma_{\varphi\to\psi}\to ({\varphi\to\psi})\Big)$. We now consider the concatenation $\tau_\varphi\star\tau_{\varphi\to\psi}$ of both $\tau$-sequences and likewise $\sigma_\varphi\star\sigma_{\varphi\to\psi}$ denotes the concatenation of both $\sigma$-sequences. Clearly, we have $|\tau_\varphi\star\tau_{\varphi\to\psi}| = |\sigma_\varphi\star\sigma_{\varphi\to\psi}|$ and $\tau_\varphi\star\tau_{\varphi\to\psi} \prec \alpha$. Likewise, from our assumptions it is easy to observe that $\la \tau_\varphi\star\tau_{\varphi\to\psi}\ra\sigma_\varphi\star\sigma_{\varphi\to\psi}$ and $\Box \Big( \la \tau_\varphi\star\tau_{\varphi\to\psi}\ra\sigma_\varphi\star\sigma_{\varphi\to\psi} \to \psi \Big)$ so that indeed $\vboxBox{\alpha}\psi$.
\end{proof}

As a consequence of our previous lemmas, we know that all reasoning of the modal logic \logic{K} can be applied to any M\"unchhausen provability predicate. We now turn to the transitivity axiom to conclude that  each predicate $\vboxBox \alpha$ actually is sound for \logic{K4}.  Before proving this, we need one easy technical observation. 

\begin{lemma}\label{theorem:BarcanForMunchhausen}
Let $T$ be a $\Lambda$-M\"unchhausen theory with corresponding predicate $\vboxBox{\alpha}$. We have that
\[
T\vdash \exists x \ \vboxBox \alpha \varphi (\dot x) \ \to \   \vboxBox \alpha \exists x \, \varphi ( x).
\]
\end{lemma}

\begin{proof}
We reason in $T$ and assume that for some $x$ we gave $\vboxBox \alpha \varphi (\dot x)$. Thus, for some (possibly empty) $\sigma$ and some ordinal $\beta$ (less than $\alpha$ in case $\sigma$ is non-empty) we have $\vboxDiamond \alpha \sigma$ and $\Box \big( \vboxDiamond \beta \sigma \to \varphi (\dot x)\big)$ whence also $\Box \big( \vboxDiamond \beta \sigma \to \exists x \varphi ( x)\big)$ as was to be shown.
\end{proof}

We can now prove the soundness of the transitivity axiom.

\begin{lemma}\label{theorem:MHTransitivityAxiom}
Let $T$ be a $\Lambda$-M\"unchhausen theory with corresponding predicate $\vboxBox{\alpha}$. We have that
\[
T\vdash \forall \, \alpha{\prec}\Lambda\, \forall\, \varphi\ \Big( \,  \vboxBox \alpha \varphi \ \to \ \vboxBox \alpha \vboxBox \alpha \varphi \,\Big). 
\]

\end{lemma}

\begin{proof}
The proof is very similar to Item \ref{item:transitivity:theorem:boxGLPSound} of Theorem \ref{theorem:boxGLPSound} but now, there is no need for induction since we already know our predicate to be sound for \logic K reasoning. Thus, we reason in $T$, fix some ordinal $\alpha \prec \Lambda$ and formula $\varphi$ and assume $\vboxBox \alpha \varphi$. Now either $\Box \varphi$ or there is some sequence of ordinals $\tau \prec \alpha$ and sequence $\sigma$ so that $\vboxDiamond \beta \sigma$ and $\Box \Big( \vboxDiamond \beta \sigma \to \varphi \Big)$. In the first case, we get from $\Box \varphi$ that $\Box \Box \varphi$ whence by applying monotonicity twice that $\vboxBox \alpha \vboxBox \alpha \varphi$. 
Thus we focus on the second case and fix a particular sequences $\tau$ and $\sigma$ so that 
\begin{enumerate}
\item
$\tau \prec \alpha$;

\item
$\vboxDiamond \tau \sigma$;

\item
$\Box \Big( \vboxDiamond \tau \sigma \to \varphi \Big)$.
\end{enumerate}
From the first item, we get by assumptions on M\"unchhausen theories that $\vboxBox \alpha (\tau \prec \alpha)$. From the second item we get by negative introspection that $\vboxBox \alpha \vboxDiamond \tau \sigma$. From the third item we get $\Box \Box \Big( \vboxDiamond \tau \sigma \to \varphi \Big)$ whence $\vboxBox \alpha \Box \Big( \vboxDiamond \tau \sigma \to \varphi \Big)$. Collecting these three consequences and applying provable closure of provability under conjunctions we obtain
\[
\exists \sigma \exists \tau \ \vboxBox \alpha \Big( \tau{\prec}\alpha \, \wedge \, \vboxDiamond \tau \sigma  \, \wedge \, \Box \big( \vboxDiamond \tau \sigma \to \varphi \big)\Big)
\]
so that by Lemma \ref{theorem:BarcanForMunchhausen} we conclude
\[
\vboxBox \alpha \exists \sigma \, \exists \tau{\prec}\alpha \, \ \Big( \vboxDiamond \tau \sigma  \, \wedge \, \Box \big( \vboxDiamond \tau \sigma \to \varphi \big)\Big)
\]
which implies $\vboxBox \alpha \vboxBox \alpha \varphi$ as was to be shown.
\end{proof}

In the light of Lemma \ref{theorem:noLoebNeeded} we may now conclude arithmetical soundness for $\glp_\Lambda$ for M\"unchhausen provability.

\begin{theorem}
Let $T$ be a $\Lambda$-M\"unchhausen theory and let $\vboxBox{\alpha}_T$ be a corresponding M\"unchhausen provability predicate. Then, $\glp_\Lambda$ is sound for $T$ when the $[\alpha ]$ -modalities ($\alpha \prec \Lambda$) are interpreted as $\vboxBox{\alpha}_T$.
\end{theorem}

\begin{proof}
As always we prove by induction on a $\glp_\Lambda$ proof that if $\glp_\Lambda \vdash \varphi$, then for any arithmetical realization $*$ we have that $T \vdash \varphi^*$.
\end{proof}

It is clear how the completeness proof and formalisation can be adapted to the new provability notion. Actually, it seems that in a sense M\"unchhausen provability is more fundamental than one-M\"unchhausen provability. We have chosen to start this paper with one-M\"unchhausen provability instead for two reasons. Firstly, the defining recursion for one-M\"unchhausen provability is slightly easier and more perspicuous. But secondly, it is important to be aware of the tension between provable properties and provable provable properties in the notion of one-M\"unchhausen provability and how this tension can be mitigated via transfinite reflexive induction.


\bibliographystyle{plain}
\bibliography{References}

\end{document}